\documentclass{elsarticle}

\usepackage[english]{babel}
\usepackage{dsfont} 
\usepackage{graphicx}
\usepackage{color}
\usepackage{hyperref}
\usepackage{enumerate}
\usepackage{amssymb,empheq} 
\usepackage{amsthm} 
\usepackage{cleveref}
\usepackage{xcolor}
\usepackage{subcaption}

\allowdisplaybreaks
\newtheorem{thm}{Theorem}[section]

\newtheorem{lem}[thm]{Lemma}
\newtheorem{prop}[thm]{Proposition}
\newtheorem{cor}[thm]{Corollary}

\newtheorem{open}[thm]{Open Problem}

\theoremstyle{definition}

\newtheorem{dfn}[thm]{Definition}
\newtheorem{exm}[thm]{Example}
\newtheorem{exms}[thm]{Examples}

\theoremstyle{remark}
\newtheorem{rem}[thm]{Remark}

\newcommand{\exmsymbol}{\hfill$\circ$}

\newcommand{\cset}{\mathds{C}}

\newcommand{\nset}{\mathds{N}}

\newcommand{\qset}{\mathds{Q}}
\newcommand{\rset}{\mathds{R}}

\newcommand{\zset}{\mathds{Z}}

\newcommand{\diff}{\mathrm{d}}

\newcommand{\pos}{\mathrm{Pos}}

\newcommand{\supp}{\mathrm{supp}\,}

\newcommand{\one}{\mathds{1}}

\newcommand{\cC}{\mathcal{C}}

\newcommand{\cH}{\mathcal{H}}

\newcommand{\cS}{\mathcal{S}}

\newcommand{\cV}{\mathcal{V}}

\newcommand{\fD}{\mathfrak{D}}
\newcommand{\fd}{\mathfrak{d}}

\newcommand{\fg}{\mathfrak{g}}

\author{Philipp J.\ di Dio}
\address{Department of Mathematics and Statistics, University of Konstanz, Universit\"atsstra{\ss}e 10, D-78464 Konstanz, Germany}
\address{Zukunftskolleg, Universtity of Konstanz, Universit\"atsstra{\ss}e 10, D-78464 Konstanz, Germany}
\address{philipp.didio@uni-konstanz.de}

\author{Konrad Schm\"udgen}
\address{Mathematical Institute, University of Leipzig, Augustusplatz 10, D-04109 Leipzig, Germany}
\address{schmuedgen@uni-leipzig.de}

\journal{arXiv}

\title{$K$-Positivity Preservers and their Generators}

\begin{document}

\begin{abstract}
We study $K$-positivity preservers with given closed $K\subseteq\rset^n$, i.e., linear maps $T:\rset[x_1,\dots,x_n]\to\rset[x_1,\dots,x_n]$ such that $T\pos(K)\subseteq\pos(K)$ holds, and their generators $A:\rset[x_1,\dots,x_n]\to\rset[x_1,\dots,x_n]$, i.e., $e^{tA}\pos(K)\subseteq\pos(K)$ holds for all $t\geq 0$.
We characterize these maps $T$ for any closed $K\subseteq\rset^n$ in \Cref{thm:KposDescription}.
We characterize the maps $A$ in \Cref{thm:PosGeneratorsGeneral} for $K=\rset^n$ and give partial results for general $K$.
In \Cref{thm:sigma} and \ref{thm:1} we give maps $A$ such that $e^{tA}$ is a positivity preserver for all $t\geq \tau$ for some $\tau>0$ but not for $t\in (0,\tau)$, i.e., we have an eventually positive semi-group.
\end{abstract}


\maketitle



\section{Introduction}

The non-negative polynomials
\[\pos(K) := \{f\in\rset[x_1,\dots,x_n] \,|\, f\geq 0\ \text{on}\ K\}\]
on closed $K\subseteq\rset^n$, $n\in \nset$, are extensively studied in the literature.
Various types of Positivstellens\"{a}tze of real algebraic geometry provide powerful descriptions of non-negative or positive polynomials on semi-algebraic sets $K$ \cite{marshallPosPoly}.

In contrast, linear maps
\begin{equation}\label{eq:linT}
T:\rset[x_1,\dots,x_n]\to\rset[x_1,\dots,x_n],
\end{equation}
which satisfy 
\begin{equation}\label{Tpos}
T\pos(K)\subseteq\pos(K)
\end{equation}
are studied much less and mainly for $K=\rset^n$, see e.g.\ \cite{guterman08,netzer10,borcea11}.
A linear map $T$ satisfying (\ref{Tpos}) is called a \emph{$K$-positivity preserver}.

The linear operators $T$ can be nicely described as differential operators of infinite order, i.e., for each $\alpha\in\nset_0^n$ there exists a unique polynomial $q_\alpha\in\rset[x_1,\dots,x_n]$ such that
\begin{equation}\label{diffT}
T = \sum_{\alpha\in\nset_0^n} q_\alpha\cdot\partial^\alpha.
\end{equation}
This result is folklore.
A proof can e.g.\ be found in \cite[Lemma\ 2.3]{netzer10}.
The representation (\ref{diffT}) is called the \emph{canonical representation} of $T$.

The main objects of this paper are $K$-positivity preservers with closed $K\subseteq\rset^n$ and generators of semi-groups of $K$-positivity preservers.
That is, we study the following problems:
\begin{enumerate}[I)]
\item When is an operator $T$ of the form (\ref{diffT}) a $K$-positivity preserver? 

\item For which linear map $A:\rset[x_1,\dots,x_n]\to\rset[x_1,\dots,x_n]$ is   $e^{tA}$  a $K$-positivity preserver for all $t\geq 0$?
\end{enumerate}

For differential operators $T$ with $K=\rset^n$ problem (I) has been solved in \cite{borcea11} and with constant coefficients $q_\alpha\in\rset$ problem (II) has been solved in \cite{didio24posPresConst}.
The present paper treats operators $T$ with polynomial coefficients $q_\alpha\in\rset[x_1,\dots,x_n]$ and general closed $K\subseteq\rset^n$.
This is technically much more involved.

The importance of the second problem is motivated by the following.
For any $f_0\in\rset[x_1,\dots,x_n]$ we have that $e^{tA}f_0$ is the unique solution of
\begin{equation}\label{eq:pde}
\partial_t f = A f \quad\text{with}\quad f(\,\cdot\,,0) = f_0.
\end{equation}
If $f_0\in\pos(K)$ for some closed subset $K\subseteq\rset^n$, it is natural to ask whether the solution $f$ of (\ref{eq:pde}) is also in $\pos(K)$ for all $t\geq 0$.
If we require this for arbitrary $f_0\in\rset[x_1,\dots,x_n]$ then it is equivalent to $e^{tA}\pos(K)\subseteq\pos(K)$ for all $t\geq 0$.
This is one reason for the interest in  operators $A$ which satisfy $e^{tA}\pos(K)\subseteq\pos(K)$  for all $t\geq 0$.

One easily sees that when optimizing a linear functional $L:\rset[x_1,\dots,x_n]\to\rset$ over some set $\cC\subseteq\rset[x_1,\dots,x_n]$ we can apply $e^{tA}$ and $e^{-tA^*}$ to get
\[L(p) = L(e^{-tA} e^{tA}p) = (e^{-tA^*} L)(e^{tA} p).\]
Hence, if $A$ and $t\geq 0$ are chosen such that $\tilde{\cC} := e^{tA}\cC$ consists only of sums of squares we have to optimize $\tilde{L} := e^{-tA^*}L$ over the subset $\tilde{\cC}$ of sums of squares:
\[\min_{p\in\cC} L(p) = \min_{p\in\tilde{\cC}} \tilde{L}(p).\]
Optimizing over (a subset of) sums of squares is computationally much more efficient.
In \cite{curtoHeat23} we observed that under the heat equation several non-negative polynomials which are not sums of squares, e.g.\ the Motzkin, the Robinson, and the Choi--Lam polynomial, become a sum of squares and in \cite[Thm.\ 3.20]{curtoHeat23} we found that every non-negative polynomial $p\in\rset[x_1,x_2,x_3]_{\leq 4}$ becomes a sum of squares after a time $\tau>0$ (independent on $p$).

For the study of the generator problem we shall define and use a non-commutative regular Fr\'echet Lie group.

\begin{dfn}
We define
\[\fD := \left\{ T = \sum_{\alpha\in\nset_0^n} q_\alpha\cdot\partial^\alpha \;\middle|\; q_\alpha\in\rset[x_1,\dots,x_n]_{\leq |\alpha|},\alpha\in\nset_0^n,\ \text{and}\ \ker T = \{0\}\right\}\]
and
\[\fd := \left\{ \sum_{\alpha\in\nset_0^n} q_\alpha\cdot\partial^\alpha \;\middle|\; q_\alpha\in\rset[x_1,\dots,x_n]_{\leq |\alpha|},\alpha\in\nset_0^n\right\}.\]
We define $\one$ by $q_0 = 1$ and $q_\alpha=0$ otherwise.
\end{dfn}

Clearly, the linear mappings $T\in\fD$ leave the vector spaces $\rset[x_1,\dots,x_n]_{\leq d}$ invariant for all $d\in\nset_0$.
This property is necessary such that $e^A$ is well-defined as an operator $\rset[x_1,\dots,x_n]\to\rset[x_1,\dots,x_n]$.

The paper is structured as follows.
To make the paper as self-contained as possible in \Cref{sec:prelim} we collect results about positivity preserver and regular Fr\'echet Lie groups which will be needed later.
In \Cref{sec:Dfrechet} we show that $(\fD,\,\cdot\,)$ is a regular Fr\'echet Lie group with Lie algebra $(\fd,\,\cdot\,,+)$.
\Cref{sec:KposPres} is devoted to the first of the above-mentioned problems, i.e., we fully characterize $K$-positivity preservers with polynomial coefficients.
$K$-positivity preservers with constant coefficients are also studied in more depth in \Cref{sec:KposPres}.
\Cref{sec:generators} is about the second problem and contains our results about the generators of $K$-positivity preserving semi-groups.
\Cref{thm:PosGeneratorsGeneral} gives the complete characterization of generators of positivity preserving semi-groups with polynomial coefficients for $K=\rset^n$.
Here infinitely divisible measures and the Levy-Khinchin formula play a crucial role.
In \Cref{sec:eventually} we give examples of eventually positive semi-groups.
In \Cref{sec:summary} we summarize our results and mention two open problems which emerged during our investigation.

We always work on $\rset[x_1,\dots,x_n]$, i.e., we always have $n\in\nset$ from now on.
We assume the set $K\subseteq\rset^n$ is always non-empty.

\section{Preliminaries}
\label{sec:prelim}

\subsection{Positivity Preserver}

\begin{dfn}
Let
\[T = \sum_{\alpha\in\nset_0^n} q_\alpha\cdot\partial^\alpha.\]
For every $y\in\rset^n$ we define the linear map
\[T_y:\rset[x_1,\dots,x_n]\to\rset[x_1,\dots,x_n]\quad \text{by}\quad T_y := \sum_{\alpha\in\nset_0^n} q_\alpha(y)\cdot\partial^\alpha.\]
\end{dfn}

\begin{prop}[{\cite[Thm.\ 3.1]{borcea11}}]\label{thm:posPresChara}
Let
\[T = \sum_{\alpha\in\nset_0^n} q_\alpha\cdot\partial^\alpha\]
be with $q_\alpha\in\rset[x_1,\dots,x_n]$.
Then the following are equivalent:
\begin{enumerate}[(i)]
\item $T$ is a positivity preserver, i.e., $T\pos(\rset^n)\subseteq\pos(\rset^n)$.

\item $(\alpha!\cdot q_\alpha(y))_{\alpha\in\nset_0^n}$ is a moment sequence for every $y\in\rset^n$.
\end{enumerate}
If one of the equivalent conditions holds then for every $y\in\rset^n$ the map $T_y$ is given by
\[(T_y p)(x) := \int_{\rset^n} f(x+z)~\diff\mu_y(z)\]
for all $p\in\rset[x_1,\dots,x_n]$ where $\mu_y$ is a representing measure of $(\alpha!\cdot q_\alpha(y))_{\alpha\in\nset_0^n}$.
\end{prop}

%
%
%

%
%
%

We see that all characterizations are based on the constant coefficient case.
We therefore introduce the following definitions, see also \cite{didio24posPresConst}.

\begin{dfn}\label{dfn:Ds}
Let $s = (s_\alpha)_{\alpha\in\nset_0^n}$ be a real sequence.
We define
\[D(s) := \sum_{\alpha\in\nset_0^n} \frac{s_\alpha}{\alpha!}\cdot\partial^\alpha.\]
\end{dfn}

\begin{lem}
Let $s = (s_\alpha)_{\alpha\in\nset_0^n}$ and $t = (t_\alpha)_{\alpha\in\nset_0^n}$ be two real sequences.
Then
\[D(s) D(t) = D(t) D(s) = D(u)\]
for a unique real sequence $u = (u_\alpha)_{\alpha\in\nset_0^n}$ with
\begin{equation}\label{eq:dfnConvolutionSequences}
u_\alpha = \sum_{\beta\preceq\alpha} \binom{\alpha}{\beta}\cdot s_\beta\cdot t_{\alpha-\beta}
\end{equation}
for all $\alpha\in\nset_0^n$.
\end{lem}

\begin{dfn}
Let $s = (s_\alpha)_{\alpha\in\nset_0^n}$ and $t = (t_\alpha)_{\alpha\in\nset_0^n}$ be two real sequences.
We define \emph{the convolution $s*t = (u_\alpha)_{\alpha\in\nset_0^n}$ of $s$ and $t$} by (\ref{eq:dfnConvolutionSequences}).
\end{dfn}

The following collects previous results, e.g.\ from \cite{didio24hadamardArXiv,didio24posPresConst}.

\begin{cor}\label{cor:momSequAddConvMult}
Let $s = (s_\alpha)_{\alpha\in\nset_0^n}$ and $t = (t_\alpha)_{\alpha\in\nset_0^n}$ be two moment sequences with representing measures $\mu$ and $\nu$.
Then the following hold:
\begin{enumerate}[(i)]
\item For all $a,b\geq 0$ we have that $as + bt$ is a moment sequence with representing measure $a\mu + b\nu$.

\item The convolution $s*t$ is a moment sequence with representing measure $\mu*\nu$.

\item The Hadamard product $s\odot t := (s_\alpha\cdot t_\alpha)_{\alpha\in\nset_0^n}$ is a moment sequence with representing measure $\mu\odot\nu$ defined by
\[(\mu\odot\nu)(A) := \int_{\rset^{2n}} \chi_A(x_1y_1,\dots,x_ny_n)~\diff\mu(x)~\diff\nu(y)\]
for all Borel sets $A\subseteq\rset^n$ and $\chi_A$ is the characteristic function of $A$.
\end{enumerate}
\end{cor}
\begin{proof}
(i) is clear,
(ii) is \cite[Cor.\ 3.7 (ii)]{didio24posPresConst},
and (iii) is \cite[Thm.\ 2]{didio24hadamardArXiv}.
\end{proof}

\subsection{Fr\'echet Spaces, $\cset[[x_1,\dots,x_n]]$, and $\cset[x_1,\dots,x_n]^\sharp[[\partial_1,\dots,\partial_n]]$}

A Fr\'echet space is a metrizable, complete, and locally convex topological vector space $\cV$.
The vector space $\cset[[x_1,\dots,x_n]]$ of formal power series is a Fr\'echet space with the family $\{p_d\}_{d\in\nset_0}$ of semi-norms $p_s(f) = \sup_{\alpha\in\nset_0^n:|\alpha|\leq d} |c_\alpha|$ of $f = \sum_{\alpha\in\nset_0^n} c_\alpha\cdot x^\alpha$, see e.g.\ \cite[pp.\ 91--92, Ex.\ III]{treves67}.
The convergence for this  Fr\'echet space is the coordinate-wise convergence.
Similarly, the space of all sequences indexed by a countable set is a Fr\'echet space.

\begin{dfn}\label{dfn:polySharpPartial}
We define
\[\cset[x_1,\dots,x_n]^\sharp[[\partial_1,\dots,\partial_n]] := \left\{ \sum_{\alpha\in\nset_0^n} p_\alpha\cdot \partial^\alpha \,\middle|\, p_\alpha\in\cset[x_1,\dots,x_n]_{\leq |\alpha|},\alpha\in\nset_0^n\right\}.\]
\end{dfn}

\begin{lem}\label{lem:polySharpPartialFrechetSpace}
The vector space $\cset[x_1,\dots,x_n]^\sharp[[\partial_1,\dots,\partial_n]]$ is a Fr\'echet space with the coordinate-wise convergence.
\end{lem}
\begin{proof}
We have
\begin{multline*}
\cset[x_1,\dots,x_n]^\sharp[[\partial_1,\dots,\partial_n]] =\\ \left\{\sum_{\alpha\in\nset_0^n} \sum_{\beta\in\nset_0^n:|\beta|\leq |\alpha|} c_{\alpha,\beta}\cdot x^\beta\cdot\partial^\alpha \,\middle|\, c_{\alpha,\beta}\in\cset\ \text{for all}\ \alpha,\beta\in\nset_0^n\ \text{with}\ |\beta|\leq |\alpha|\right\},
\end{multline*}
i.e., $\cset[x_1,\dots,x_n]^\sharp[[\partial_1,\dots,\partial_n]]$ is as a vector space isomorphic to the vector space of sequences
\[\cS := \big\{(c_{\alpha,\beta})_{\alpha,\beta\in\nset_0^n: |\beta|\leq |\alpha|} \,\big|\, c_{\alpha,\beta}\in\cset\ \text{for all}\ \alpha,\beta\in\nset_0^n\ \text{with}\ |\beta|\leq |\alpha|\big\}.\]
But $\cS$ as the vector space of sequences with the coordinate-wise convergence is a Fr\'echet space and hence so is $\cset[x_1,\dots,x_n]^\sharp[[\partial_1,\dots,\partial_n]]$.
\end{proof}

\begin{rem}
By the same argument as in the proof of \Cref{lem:polySharpPartialFrechetSpace} we have that
\[\cset[[x_1,\dots,x_n]][[\partial_1,\dots,\partial_n]] := \left\{ \sum_{\alpha\in\nset_0^n} p_\alpha\cdot \partial^\alpha \,\middle|\, p_\alpha\in\cset[[x_1,\dots,x_n]],\alpha\in\nset_0^n \right\}\]
is a Fr\'echet space with the coordinate-wise convergence.
\exmsymbol
\end{rem}

For more about Fr\'echet spaces see e.g.\ \cite{treves67,schaef99}.

\subsection{Regular Fr\'echet Lie Groups and their Lie Algebra}

We give here only the definition of a regular Fr\'echet Lie group according to \cite{omori97}.

\begin{dfn}[see e.g.\ {\cite[p.\ 63, Dfn.\ 1.1]{omori97}}]\label{dfn:frechetLieGroup}
We call $(G,\,\cdot\,)$ a \emph{regular Fr\'echet Lie group} if the following conditions are fulfilled:
\begin{enumerate}[(i)]
\item $G$ is an infinite dimensional smooth Fr\'echet manifold.

\item $(G,\,\cdot\,)$ is a group.

\item The map $G\times G\to G$, $(A,B)\mapsto A\cdot B^{-1}$ is smooth.

\item The \emph{Fr\'echet Lie algebra} $\fg$ of $G$ is isomorphic to the tangent space $T_e G$ of $G$ at the unit element $e\in G$.

\item $\exp:\fg\to G$ is a smooth mapping such that
\[\left.\frac{\diff}{\diff t} \exp(t u)\right|_{t=0} = u\]
holds for all $u\in\fg$.

\item The space $C^1(G,\fg)$ of $C^1$-curves in $G$ coincides with the set of all $C^1$-curves in $G$ under the Fr\'echet topology.
\end{enumerate}
\end{dfn}

For more on infinite dimensional manifolds, differential calculus, Lie groups, and Lie algebras see e.g.\ \cite{omori97,schmed23}.

\subsection{Generators of Positivity Preservers with Constant Coefficients}

The following is a result from \cite{didio24posPresConst} for positivity preservers with constant coefficients and their generators.

\begin{dfn}
We define
\[\fd_0 := \left\{ \sum_{\alpha\in\nset_0^n} c_\alpha\cdot\partial^\alpha \,\middle|\, c_\alpha\in\rset\ \text{for all}\ \alpha\in\nset_0^n\ \text{and}\ c_0=0\right\}\]
and the set of all generators of positivity preservers with constant coefficients
\[\fd_{0,+} := \left\{A\in\delta_0 \,\middle|\, e^{tA}\ \text{positivity preserver for all}\ t\geq 0\right\}.\]
\end{dfn}

\begin{prop}[{\cite[Main Thm.\ 4.11]{didio24posPresConst}}]\label{thm:PosGeneratorsConstCoeffs}
The following are equivalent:
\begin{enumerate}[(i)]
\item $\displaystyle A = \sum_{\alpha\in\nset_0^n\setminus\{0\}} \frac{a_\alpha}{\alpha!}\cdot\partial^\alpha \in\fd_{0,+}$.

\item There exists a symmetric matrix $\Sigma = (\sigma_{i,j})_{i,j=1}^n\in\rset^n$ with $\Sigma\succeq 0$, a vector $b = (b_1,\dots,b_n)^T\in\rset^n$, and a measure $\nu$ on $\rset^n$ with
\[\int_{\|x\|_2\geq 1} |x_i|~\diff\nu(x)<\infty \qquad\text{and}\qquad \int_{\rset^n} |x^\alpha|~\diff\nu(x)<\infty\]
for all $i=1,\dots,n$ and $\alpha\in\nset_0^n$ with $|\alpha|\geq 2$ such that
\begin{align*}
a_{e_i} &= b_i + \int_{\|x\|_2\geq 1} x_i~\diff\nu(x) && \text{for all}\ i=1,\dots,n,\\
a_{e_i+e_j} &= \sigma_{i,j} + \int_{\rset^n} x^{e_i + e_j}~\diff\nu(x) && \text{for all}\ i,j=1,\dots,n,
\intertext{and}
a_\alpha &= \int_{\rset^n} x^\alpha~\diff\nu(x) &&\text{for all}\ \alpha\in\nset_0^n\ \text{with}\ |\alpha|\geq 3.
\end{align*}
\end{enumerate}
\end{prop}

\section{$(\fD,\,\cdot\,)$ is a regular Fr\'echet Lie Group with Lie Algebra $(\fd,\,\cdot\,,+)$}
\label{sec:Dfrechet}

We begin with some simple technical facts.

\begin{lem}\label{lem:degreePreserving}
For a linear map $T:\rset[x_1,\dots,x_n]\to\rset[x_1,\dots,x_n]$ the following are equivalent:
\begin{enumerate}[(i)]
\item $\displaystyle T = \sum_{\alpha\in\nset_0^n} q_\alpha\cdot\partial^\alpha$ holds with $\deg q_\alpha \leq |\alpha|$ for all $\alpha\in\nset_0^n$.

\item $T\rset[x_1,\dots,x_n]_{\leq d}\subseteq \rset[x_1,\dots,x_n]_{\leq d}$ holds for all $d\in\nset_0$.
\end{enumerate}
\end{lem}
\begin{proof}
(i) $\Rightarrow$ (ii):
Let $p\in\rset[x_1,\dots,x_n]$.
Then $\deg (q_\alpha\cdot\partial^\alpha p) \leq \deg p$ for all $\alpha\in\nset_0^n$ and hence $T\rset[x_1,\dots,x_n]_{\leq d}\subseteq\rset[x_1,\dots,x_n]_{\leq d}$.

(ii) $\Rightarrow$ (i): 
Assume to the contrary that there is an $\alpha\in\nset_0^n$ with $\deg q_\alpha > |\alpha|$.
Take the smallest of these $\alpha$'s  with respect to the lexicographic order.
Then $\deg (q_\alpha\cdot\partial^\alpha x^\alpha) = \deg q_\alpha > |\alpha|$ which contradicts (ii).
\end{proof}

\begin{lem}\label{lem:groupContinuousMultInv}
The following hold:
\begin{enumerate}[(i)]
\item $(\fD,\,\cdot\,)$ is a group.

\item $\cdot:\fD\times\fD\to\fD$, $(A,B)\mapsto A\cdot B$ is continuous in the Fr\'echet topology of $\fD$.

\item $\cdot^{-1}:\fD\to\fD$, $A\mapsto A^{-1}$ is continuous in the Fr\'echet topology of $\fD$.
\end{enumerate}
\end{lem}
\begin{proof}
(i): Since the multiplication in $\fD$ is the composition of maps it is associative.
Let $A,B\in\fD$. Then
\[A\rset[x_1,\dots,x_n]_{\leq d},\ B\rset[x_1,\dots,x_n]_{\leq d}\ \subseteq\ \rset[x_1,\dots,x_n]_{\leq d}\]
holds for all $d\in\nset_0$ by \Cref{lem:degreePreserving} and hence we have
\[AB\rset[x_1,\dots,x_n]_{\leq d}\ \subseteq\ \rset[x_1,\dots,x_n]_{\leq d}.\]
Additionally, we have $A1\neq 0$ and $B1\neq 0$, i.e., $AB1 \neq 0$.
Therefore, $AB\in\fD$ again by \Cref{lem:degreePreserving}.
Thus, $(\fD,\,\cdot\,)$ is a semi-group.

It is therefore sufficient to show that for any $A\in\fD$ there exists a $B\in\fD$ such that $AB = BA = \one$.
Let $A = \sum_{\alpha\in\nset_0^n} a_\alpha\cdot\partial^\alpha$ and $B = \sum_{\beta\in\nset_0^n} b_\beta\cdot\partial^\beta$.
Then
\[AB = \sum_{\alpha,\beta\in\nset_0^n} a_\alpha\partial^\alpha b_\beta\partial^\beta = C = \sum_{\gamma\in\nset_0^n} c_\gamma\partial^\gamma\in\fD.\]
Hence, $AB = \one$ is equivalent to
\begin{align}
0\neq c_0 = C1 = AB1 = a_0 b_0\qquad &\Rightarrow\qquad b_0 = a_0^{-1}\neq 0\label{eq:inverse1}\\
0 = c_\alpha = a_0 b_\alpha + \tilde{c}_\alpha\qquad &\Rightarrow\qquad b_\alpha = -\frac{\tilde{c}_\alpha}{a_0}\in\rset[x_1,\dots,x_n]_{\leq |\beta|}\label{eq:inverse2}
\end{align}
for all $\alpha\in\nset_0^n$.
Hence, $A$ has a right inverse $B$, i.e. $AB = \one$.
Similarly, $A$ has a left inverse $\tilde{B}$, i.e. $\tilde{B}A = \one$.
The associativity yields $\tilde{B} = \tilde{B}\one = \tilde{B}AB = \one B = B$, so that $A$ has an inverse in $\fD$.

(ii):
Let $(A_k)_{k\in\nset_0}$ and $(B_k)_{k\in\nset_0}$ be sequences in $\fD$ such that $A_k\to A\in\fD$ and $B_k\to B\in\fD$ as $k\to\infty$.
Since $A_k|_{\rset[x_1,\dots,x_n]_{\leq d}}$, $B_k|_{\rset[x_1,\dots,x_n]_{\leq d}}$, $A|_{\rset[x_1,\dots,x_n]_{\leq d}}$, and $B|_{\rset[x_1,\dots,x_n]_{\leq d}}$ are bounded operators, we get $A_k B_k p \to ABp$ for $k\to\infty$ and all $p\in\rset[x_1,\dots,x_n]$, i.e., $A_k B_k \to AB$ in the Fr\'echet topology of $\fD$. Hence, the multiplication is continuous.

(iii):
Let $(A_k)_{k\in\nset_0}$ be a sequence in $\fD$ such that $A_k\to A\in\fD$ as $k{\to}\infty$. Then,  $a_{k,0} = A_k1 \to A1 = a_0\neq 0$.
Hence, we have $A_k^{-1}\to A^{-1}$ since the equations (\ref{eq:inverse1}) and (\ref{eq:inverse2}) for the polynomial coefficients of the inverse are continuous.
\end{proof}

\begin{lem}
 $(\fd,\,\cdot\,,+)$ is an algebra.
\end{lem}
\begin{proof}
It is sufficient to show that $AB\in\fd$ holds for all $A,B\in\fd$.
Let $A,B\in\fd$. 
By \Cref{lem:degreePreserving} we have
\[A\rset[x_1,\dots,x_n]_{\leq d},\ B\rset[x_1,\dots,x_n]_{\leq d}\ \subseteq\ \rset[x_1,\dots,x_n]_{\leq d}\]
for all $d\in\nset_0$ and hence
\[AB\rset[x_1,\dots,x_n]_{\leq d}\ \subseteq\ \rset[x_1,\dots,x_n]_{\leq d}.\]
This gives $AB\in\fd$ by \Cref{lem:degreePreserving}.
\end{proof}

\begin{lem}\label{lem:expMapSmoothTangent}
The exponential map
\[\exp:\fd\to\fD,\quad A\mapsto \exp A := \sum_{k\in\nset_0} \frac{A^k}{k!}\]
is smooth and
\[\frac{\diff}{\diff t} \exp(tA) \Big|_{t=0} = A\]
holds for all $A\in\fd$.
\end{lem}
\begin{proof}
Let $A\in\fd$.
For $d\in\nset_0$ we define $A_d := A|_{\rset[x_1,\dots,x_n]_{\leq d}}$, i.e., $A_d$ is a bounded operator (matrix) and $\exp A_d$ is well-defined.
We have that
\[(\exp A)p = \sum_{k\in\nset_0} \frac{A^k}{k!}p = \sum_{k\in\nset_0} \frac{A^k_d}{k!} p = (\exp A_d)p\]
holds for all $p\in\rset[x_1,\dots,x_n]$ and $d\geq \deg p$.
This also shows that $\exp$ is smooth since it is smooth for the bounded operator (matrix) case.
Additionally, we have
\[\frac{\diff}{\diff t} \exp(tA) \Big|_{t=0} p = \frac{\diff}{\diff t} \exp(tA_d) \Big|_{t=0} p = A_d p = Ap\]
for all $p\in\rset[x_1,\dots,x_n]$ with $d \geq \deg p$.
\end{proof}

In the previous proof we actually proved the following as a byproduct.

\begin{cor}\label{cor:Ad}
Let $A\in\fd$ and $d\in\nset_0$.
Set $A_d := A|_{\rset[x_1,\dots,x_n]_{\leq d}}$.
Then
\[\exp(tA)p = \exp(tA_d)p\]
holds for all $t\in\rset$ and $p\in\rset[x_1,\dots,x_n]_{\leq d}$.
\end{cor}

From \Cref{cor:Ad} we obtain for the matrix $\exp$-function the following.

\begin{cor}\label{cor:expLimAd}
Let $A\in\fd$.
Then we have
\[\exp A = \lim_{k\to\infty} \left(\one + \frac{A}{k}\right)^k = \lim_{k\to\infty} \left(\one - \frac{A}{k}\right)^{-k}.\]
\end{cor}
\begin{proof}
Let $p\in\rset[x_1,\dots,x_n]$ and $d\geq \deg p$.
Again we set $A_d := A|_{\rset[x_1,\dots,x_n]_{\leq d}}$.
Since $A_d:\rset[x_1,\dots,x_n]_{\leq d}\to\rset[x_1,\dots,x_n]_{\leq d}$ is a bounded operator (matrix) we have
\[\exp A_d = \lim_{k\to\infty} \left( \one + \frac{A_d}{k}\right)^k = \lim_{k\to\infty} \left( \one - \frac{A_d}{k}\right)^{-k}.\]
Hence, by \Cref{cor:Ad} we have
\[(\exp A) p = (\exp A_d) p = \lim_{k\to\infty} \left(\one + \frac{A_d}{k}\right)^k p = \lim_{k\to\infty} \left( \one + \frac{A}{k}\right)^k p\]
and
\[(\exp A) p = (\exp A_d) p = \lim_{k\to\infty} \left(\one - \frac{A_d}{k}\right)^{-k} p = \lim_{k\to\infty} \left( \one - \frac{A}{k}\right)^{-k} p.\]
Since $p\in\rset[x_1,\dots,x_n]$ was arbitrary the assertion is proved.
\end{proof}

In summary, we have proved the following theorem.

\begin{thm}\label{thm:DregularFrechetLieGroup}
$(\fD,\,\cdot\,)$ is a regular Fr\'echet Lie group with the (Fr\'echet) Lie algebra $(\fd,\,\cdot\,,+)$.
The exponential map is given by
\[\exp:\fd\to\fD,\quad A\mapsto \exp A := \sum_{k\in\nset_0} \frac{A^k}{k!},\]
it is smooth, and it fulfills
\[\frac{\diff}{\diff t} \exp(tA) \Big|_{t=0} = A\]
for all $A\in\fd$.
\end{thm}
\begin{proof}
We have to check that the conditions (i) -- (vi) in \Cref{dfn:frechetLieGroup} are fullfilled.
(i) is clear.
(ii) and (iii) are \Cref{lem:groupContinuousMultInv}.
(iv) and (v) are \Cref{lem:expMapSmoothTangent}.

(vi): Let us consider a curve $F:\rset\to \fD$, $F(t) = \sum_{\alpha\in\nset_0^n} F_\alpha(t)\cdot\partial^\alpha$ with $F_\alpha:\rset\to\rset[x_1,\dots,x_n]_{\leq |\alpha|}$. Clearly, this curve is $C^m$, $m\in\nset_0$, if and only if all  $F_\alpha$ are $C^m$.
The functions $F_\alpha$ being $C^m$ are well-defined since all $\rset[x_1,\dots,x_n]_{\leq |\alpha|}$ are finite dimensional and hence have a unique (Euclidean) topology.
\end{proof}

\section{$K$-Positivity Preserver}
\label{sec:KposPres}

We begin with a simple method to construct non-constant positivity preservers.

\begin{exm}
Let $p=(p_1,\dots,p_n)\in\rset[x_1,\dots,x_n]^n$ and let $s = (s_\alpha)_{\alpha\in\nset_0^n}$ be a moment sequence.
Then
\[T = \sum_{\alpha\in\nset_0^n} \frac{p^{\alpha}\cdot s_\alpha}{\alpha!}\cdot\partial^\alpha\]
is a positivity preserver with non-constant coefficients.
\exmsymbol
\end{exm}
\begin{proof}
For any $y\in\rset^n$ we have that $p_y := (p^\alpha(y))_{\alpha\in\nset_0^n}$ is a moment sequence with representing measure $\delta_{p(y)}$. 
Hence, by \Cref{cor:momSequAddConvMult} (iii) we have for every $y\in\rset^n$ that $(p^\alpha(y)\cdot s_\alpha)_{\alpha\in\nset_0^n}$ is a moment sequence.
Therefore, by \Cref{thm:posPresChara} we have that $T$ is a positivity preserver.
\end{proof}

From the previous example we see that the degree of $q_\alpha$ can be much larger than $|\alpha|$.
The following result shows that the degree of the polynomial coefficients $q_\alpha$ of a positivity preserver $T$ can even grow arbitrary large.
It is sufficient to show this for the case $n=1$.

\begin{prop}
Let $(r_i)_{i\in\nset_0}\subseteq\nset$ be a sequence of natural numbers.
Then there exist a sequence $(c_i)_{i\in\nset_0}\subseteq (0,\infty)$ of positive numbers and a sequence $(k_i)_{i\in\nset_0}\subseteq\nset$ of natural numbers such that
\[T = \sum_{i\in\nset_0} p_i\cdot\partial_x^i\]
is a positivity preserver with
\[p_{2i}(x) = c_i + x^{2k_i r_i} \quad\text{and}\quad p_{2i+1}(x) = 0\]
for all $i\in \nset_0$.
\end{prop}
\begin{proof}
By \Cref{thm:posPresChara} it is sufficient to show that $p(y) := (p_i(y))_{i\in\nset_0}$ is a moment sequence for all $y\in\rset$.
Since we are in the one-dimensional case this is equivalent to $\cH_i(p(y))\succeq 0$ for all Hankel matrices $\cH_i(p(y)) := (p_{j+l}(y))_{j,l=0}^i$, $i\in\nset_0$, see e.g.\ \cite[Thm.\ 3.8 and Prop.\ 3.11]{schmudMomentBook}.
But for this it is sufficient to ensure that
\begin{equation}\label{eq:hankeldet}
\det \cH_i(p(y))>0
\end{equation}
for all $i\in\nset_0$ and $y\in\rset$.
We construct the sequences $(c_i)_{i\in\nset_0}$ and $(k_i)_{i\in\nset_0}$ from (\ref{eq:hankeldet}) by induction:

$i=0$:
We have $\det \cH_0(p(y)) = p_0(y) = c_0 + y^{2k_0 r_0} \geq 1$ for $c_0=1$ and $k_0=1$.

$i\to i+1$:
Assume we have $(c_j)_{j=0}^i$ and $(k_j)_{j=0}^i$ such that $\det \cH_j(p(y))\geq 1$ holds for all $j=0,\dots,i$ and all $y\in\rset$.
Since
\[\det \cH_{i+1}(p(y)) = p_{2i+2}(y)\cdot\det\cH_i(p(y)) + q_i(p_0,p_2,\dots,p_{2i})\]
holds by developing the determinant $\det\cH_{i+1}(p(y))$ we can choose $c_{i+1}\gg 1$ and $k_{i+1}\gg 1$ such that
\[c_{i+1} \geq 1 + \max_{x\in [-2,2]} |q_i(p_0(x),p_2(x),\dots,p_{2i}(x))|\]
and
\[x^{2k_{i+1} r_{i+1}} \geq 1 + |q_i(p_0(x),p_2(x),\dots,p_{2i}(x)|\]
hold for all $x\in (-\infty,-2]\cup [2,\infty)$ since $\det\cH_i(p(y))\geq 1$.
\end{proof}

The following is a necessary condition that a linear map $T$
is a $K$-positivity preserver.

\begin{lem}\label{lem:KposMomSeq}
Let $q_\alpha\in\rset[x_1,\dots,x_n]$ for $\alpha\in\nset_0^n$.
Suppose that
\[T = \sum_{\alpha\in\nset_0^n} q_\alpha\cdot\partial^\alpha\]
is a $K$-positivity preserver, $K\subseteq\rset^n$ closed.
Then
\[(\alpha!\cdot q_\alpha(y))_{\alpha\in\nset_0^n}\]
is a $(K-y)$-moment sequence for all $y\in K$.
\end{lem}
\begin{proof}
Since $T$ is a $K$-positivity preserver we have that
\[L_y:\rset[x_1,\dots,x_n]\to\rset,\quad f\mapsto L_y(f) := (T_y f)(y)\]
is a $K$-moment sequence, i.e., there exists a Radon measure $\mu_y$ with $\supp\mu_y\subseteq K$ such that
\[\sum_{\alpha\in\nset_0^n} q_\alpha(y)\cdot f^{(\alpha)}(y) = (T_y f)(y) = L_y(f) = \int_K f(x)~\diff\mu_y(x)\]
holds for all $f\in \rset[x_1,\dots,x_n]$.
Using the Taylor formula
\[f(x) = \sum_{\alpha\in\nset_0^n} \frac{f^{(\alpha)}(y)}{\alpha!}\cdot (x-y)^\alpha\]
we obtain
\[\sum_{\alpha\in\nset_0^n} q_\alpha(y)\cdot f^{(\alpha)}(y) = \sum_{\alpha\in\nset_0^n} \frac{f^{(\alpha)}(x)}{\alpha!}\cdot \int (x-y)^\alpha~\diff\mu_y(x)\]
for all $f\in\rset[x_1,\dots,x_n]$.
For $\beta\in\nset_0^n$ set $f_\beta(x) := (x-y)^\beta$.
Then $f^{(\alpha)}_{\beta}(y) = \beta!\cdot \delta_{\alpha,\beta}$ and hence 
\[\beta!\cdot q_\beta(y) = \int_K (x-y)^\beta~\diff\mu_y(x) = \int_{K-y} z^\beta~\diff\mu_y(z+y),\]
i.e., $(\alpha!\cdot q_\alpha(y))_{\alpha\in\nset_0^n}$ is a $(K-y)$-moment sequence with representing measure $\nu_y(\,\cdot\,) := \mu_y(\,\cdot\,+y)$.
\end{proof}

\begin{lem}\label{lem:K-yMomSeqMeasure}
Let $T = \sum_{\alpha\in\nset_0^n} q_\alpha\cdot \partial^\alpha$ with $q_\alpha\in\rset[x_1,\dots,x_n]$ for $\alpha\in\nset_0^n$ and let $K\subseteq\rset^n$ be closed.
If $y\in K$ is such that $(\alpha!\cdot q_\alpha(y))_{\alpha\in\nset_0^n}$ is a moment sequence with representing measure $\mu_y$ then
\[(T_y f)(y) = \int f(x+y)~\diff\mu_y(x), \quad f\in\rset[x_1,\dots,x_n].\]
\end{lem}
\begin{proof}
Since $\mu_y$ is a representing measure of $(\alpha!\cdot q_\alpha(y))_{\alpha\in\nset_0^n}$ we have that
\[\alpha!\cdot q_\alpha(y) = \int x^\alpha~\diff\mu_y(x)\]
holds for all $\alpha\in\nset_0^n$.
From Taylor's formula
\[f(x+y) = \sum_{\alpha\in\nset_0^n} \frac{f^{(\alpha)}(y)}{\alpha!}\cdot x^\alpha\]
we obtain
\begin{multline*}
\int f(x+y)~\diff\mu_y(x) = \sum_{\alpha\in\nset_0^n} \frac{f^{(\alpha)}(y)}{\alpha!}\cdot \int x^\alpha~\diff\mu_y(x)\\
= \sum_{\alpha\in\nset_0^n} f^{(\alpha)}(y)\cdot p_\alpha(y) = (T_y f)(y).\qedhere
\end{multline*}
\end{proof}

Now we put both lemmas together to obtain a complete characterization of $K$-positivity preserver for any closed $K\subseteq\rset^n$.
It is the general version of \Cref{thm:posPresChara} and solves a question open since \cite{netzer10,borcea11}.

\begin{thm}\label{thm:KposDescription}
Let
\[T = \sum_{\alpha\in\nset_0^n} q_\alpha\cdot\partial^\alpha,\]
be with $q_\alpha\in\rset[x_1,\dots,x_n]$ for all $\alpha\in\nset_0^n$ and let $K\subseteq\rset^n$ be closed.
Then the following are equivalent:
\begin{enumerate}[(i)]
\item $T$ is a $K$-positivity preserver.

\item For all $y\in K$ the sequence $(\alpha!\cdot q_\alpha(y))_{\alpha\in\nset_0^n}$ is a $(K-y)$-moment sequence.
\end{enumerate}
If one of the equivalent conditions (i) or (ii) holds and $\mu_y$ is a representing measure of the $(K-y)$-moment sequence $(\alpha!\cdot q_\alpha(y))_{\alpha\in\nset_0^n}$ with $y\in K$ then
\[(T_y f)(y) = \int f(x+y)~\diff\mu_y(x)\]
holds for all $f\in\rset[x_1,\dots,x_n]$.
\end{thm}
\begin{proof}
(i) $\Rightarrow$ (ii):
This is \Cref{lem:KposMomSeq}.

(ii) $\Rightarrow$ (i): 
Let $y\in K$.
Since $(\alpha!\cdot q_\alpha(y))_{\alpha\in\nset_0^n}$ is a $(K-y)$-moment sequence it has a representing measure $\mu_y$ with $\supp\mu_y\subseteq K-y$.
Let $f\in\pos(K)$.
Then, by \Cref{lem:K-yMomSeqMeasure} we have that
\[(Tf)(y) = (T_y f)(y) = \int_{K-y} f(x+y)~\diff\mu_y(x) = \int_K f(z)~\diff\mu_y(z-y) \geq 0\]
since $f\geq 0$ on $K$ and $\supp\mu_y(\,\cdot\,-y)\subseteq K$.
Since $f\in\pos(K)$ and $y\in K$ were arbitrary we have $T\pos(K)\subseteq\pos(K)$.
Hence, (i) is proved.
\end{proof}

The next lemma contains a topological fact about $K$-positivity preservers.

\begin{lem}\label{lem:closedPosPres}
Let $K\subseteq\rset^n$ be closed.
Then the following hold:
\begin{enumerate}[(i)]
\item The set of $K$-moment sequences is closed in the Fr\'echet topology of $\rset^{\nset_0^n}$.

\item The set of $K$-positivity preservers $T$ such that $$T\rset[x_1,\dots,x_n]_{\leq d}\subseteq\rset[x_1,\dots,x_n]_{\leq d}$$  for  $d\in\nset_0$ is closed in the Fr\'echet topology of $\cset[x_1,\dots,x_n]^\sharp [[\partial_1,\dots,\partial_n]]$.
\end{enumerate}
\end{lem}
\begin{proof}
(i): The convergence in the Fr\'echet topology of $\rset^{\nset_0^n}$ is the coordinate-wise convergence of sequences.
Hence, let $s_k = (s_{k,\alpha})_{\alpha\in\nset_0^n}$ be $K$-moment sequences for all $k\in\nset$ such that $s_k\to s_0$ in $\rset^{\nset_0^n}$, i.e., $s_{k,\alpha}\to s_{0,\alpha}$ for all $\alpha\in\nset_0^n$.
We have to show that $s_0$ is a $K$-moment sequence.

Let $p\in\rset[x_1,\dots,x_n]$ be a polynomial such that $p(x)\geq 0$ for all $x\in K$.
Then
\[L_{s_0}(p) = \lim_{k\to\infty} L_{s_k}(p) \geq 0\]
which implies that $s_0$ is a $K$-moment sequence by Haviland's Theorem \cite{havila36}.

(ii): Let $T_k = \sum_{\alpha\in\nset_0^n} q_\alpha\cdot\partial^\alpha$ be $K$-positivity preservers for all $k\in\nset$ with
\begin{equation}\label{eq:proofTkConvergence}
T_k\to T_0\in\cset[x_1,\dots,x_n]^\sharp [[\partial_1,\dots,\partial_n]].
\end{equation}
Then for all $y\in K$ we have that $s_k(y) := (\alpha!\cdot q_{k,\alpha}(y))_{\alpha\in\nset_0^n}$ is a $(K-y)$-moment sequence by \Cref{thm:posPresChara}.
By (\ref{eq:proofTkConvergence}) we have $s_{k,\alpha}(y)\to s_{0,\alpha}$ for all $\alpha\in\nset_0^n$ and by (i) we therefore have that $s_0$ is a $(K-y)$-moment sequence.
Since $y\in K$ was arbitrary, it follows from   \Cref{thm:posPresChara} that $T_0$ is a $K$-positivity preserver.
\end{proof}

\Cref{thm:KposDescription} gives a characterization of general $K$-positivity preservers by a necessary and sufficient condition.
Now we turn to a deeper study of $K$-positivity preservers with \textit{constant} coefficients.
For this the following definition is crucial.

\begin{dfn}\label{dfn:ksharp}
For $K\subseteq\rset^n$ we define
\[K^\sharp := \{x\in\rset^n \,|\, x+K\subseteq K\}.\]
\end{dfn}

Some simple properties of $K^\sharp$ are collected in the following lemma.
All statements follow immediately from \Cref{dfn:ksharp}.

\begin{lem}\label{lem:KsharpProps}
Let $K\subseteq\rset^n$ be closed.
Then the following hold:
\begin{enumerate}[(i)]
\item $0\in K^\sharp \neq \emptyset$.

\item $(y+K)^\sharp = K^\sharp$ for all $y\in\rset^n$.

\item $x+y\in K^\sharp$ for all $x,y\in K^\sharp$.
\end{enumerate}
\end{lem}

Depending on $K$ the set $K^\sharp$ can be very diverse.
Though $K^\sharp$ is invariant under addition by \Cref{lem:KsharpProps} (iii) it is in general not a cone, as the following examples show.

\begin{exms}\label{exm:Ksharp}
\begin{enumerate}[\bfseries\; (a)]
\item If $K$ is compact then  $K^\sharp = \{0\}$.

\item Let $c\in\rset$ and let $K = [-1,1]\times [c,\infty)$.
Then
\[K^\sharp = \{0\}\times [0,\infty).\]

\item For a closed convex cone $K\subseteq\rset^n$ with apex $0$ we have $K^\sharp = K$.

\item Let $r\in [0,1/2]$ and let $K = \bigcup_{z\in\zset^n} B_{r}(z)$. Then $K^\sharp = \zset^n$.
\end{enumerate}
\end{exms}

\begin{lem}\label{lem:sKsharpMomSeq}
Let $K\subseteq\rset^n$ be closed.
If $s$ is a $K^\sharp$-moment sequence then $D(s)$ is a $K$-positivity preserver.
\end{lem}
\begin{proof}
Since $s$ is a $K^\sharp$-moment sequence it has a representing measure $\mu$.
Let $f\in\pos(K)$ and $y\in K$.
Then
\[(D(s)f)(y) = \int f(x+y)~\diff\mu(x) \geq 0\]
holds since $\supp\mu \subseteq K^\sharp$ and hence $x+y\in K$ holds for all $x\in K^\sharp$. This shows that $D(s)$ is a $K$-positivity preserver.
\end{proof}

\begin{dfn}
Let $K\subseteq\rset^n$ be closed.
A $K$-positivity preserver
\[T = \sum_{\alpha\in\nset_0^n} q_\alpha\cdot\partial^\alpha\]
is called \emph{determinate} if all moment sequences $(\alpha!\cdot q_\alpha(y))_{\alpha\in\nset_0^n}$ are determinate for $y\in K$, i.e., have a single representing measure on $\rset^n$.
\end{dfn}

Note, we require that all $(\alpha!\cdot q_\alpha(y))_{\alpha\in\nset_0^n}$ are determinate, i.e., on all $\rset^n$ and not just on $K$.
There are moment sequences which are e.g.\ determinate on $[0,\infty)$ (Stieltjes) but are not determinate on $\rset$ (Hamburger).

\begin{thm}\label{thm:determinateKsharp}
Let $K\subseteq\rset^n$ be closed and let $s=(s_\alpha)_{\alpha\in\nset_0^n}$ be a real sequence.
Then the following statements are equivalent:
\begin{enumerate}[(i)]
\item $D(s)$ is a determinate $K$-positivity preserver.

\item $s$ is a determinate  $K^\sharp$-moment sequence.
\end{enumerate}
\end{thm}
\begin{proof}
(i) $\Rightarrow$ (ii):
By \Cref{thm:KposDescription},  $s$ is a $(K-y)$-moment sequence for all $y\in K$.
Since $D(s)$ is determinate by (i) the unique representing measure $\mu$ of $s$ satisfies $\supp\mu \subseteq K-y$ for all $y\in K$, i.e., $y+\supp\mu \subseteq K$ holds for all $y\in K$ and therefore $\supp\mu\subseteq K^\sharp$.

(ii) $\Rightarrow$ (i): This is \Cref{lem:KsharpProps}.
\end{proof}

The following corollary shows that for compact sets $K$ there are only trivial $K$-positivity preservers with constant coefficients.

\begin{cor}\label{cor:compactK}
Let $K\subseteq\rset^n$ be compact.
Then every $K$-positivity preserver $T = \sum_{\alpha\in\nset_0^n} q_\alpha \cdot\partial^\alpha$ with constant coefficients $q_\alpha \in \rset$, $\alpha\in \nset_0^n$, is of the form $T = c\cdot \one$ for some $c\in [0,\infty)$.
\end{cor}
\begin{proof}
Let $y\in K$.
Since $T$ is a $K$-positivity preserver it follows from \Cref{thm:KposDescription} that $s:= (\alpha!\cdot q_{\alpha})_{\alpha\in\nset_0^n}$ is a $(K-y)$-moment sequence. Because $K$ is compact so is $(K-y)$.
Hence, $s$ is determinate and $T$ is a determinate positivity preserver.
Then \Cref{thm:determinateKsharp} applies and we have that $s$ is a $K^\sharp$-moment sequence.
But $K^\sharp = \{0\}$ because $K$ is compact. 
Hence, $c\cdot\delta_0$ with $c\geq 0$ are the only possible representing measures of $s$.
This implies that $T=c\cdot\one$.
\end{proof}

In order to formulate the next corollary we recall the Carleman condition.
For an $n$-sequence $s = (s_\alpha )_{\alpha\in \nset_0^n}$ the marginal sequences are the $1$-sequences
\[s^{[1]}:=(s_{(k,0,\dots,0))})_{k \in \nset_0},\ s^{[2]}:=(s_{(0,k,\dots,0)})_{k \in \nset_0},\ \dots,\ s^{[n]}:=(s_{(0,\dots,0,k)})_{k \in \nset_0}.\]

\begin{dfn}
Let $s$ be a positive semidefinite $n$-sequence.
We shall say that $s$ satisfies the \emph{multivariate Carleman condition} if
\[\sum_{k=1}^\infty \left(s_{2k}^{[j]}\right)^{-\frac{1}{2k}} =+\infty\]
holds for all $j=1,\dots,n$.
\end{dfn}

\begin{cor}\label{cor:KwithCarleman}
Let $s$ be a positive semi-definite real sequence which fulfills the multivariate Carleman condition, $K\subseteq\rset^n$ closed.
The following are equivalent:
\begin{enumerate}[(i)]
\item $D(s)$ is a $K$-positivity preserver.

\item $s$ is a $K^\sharp$-moment sequence.
\end{enumerate}
\end{cor}
\begin{proof}
Since $s$ fulfills the Carleman condition we have that $s$ is a determinate moment sequence \cite[Theorem 14.19]{schmudMomentBook}.
Hence, the equivalence (i) $\Leftrightarrow$ (ii) follows from \Cref{thm:determinateKsharp}.
\end{proof}

In \Cref{lem:sKsharpMomSeq} we have seen that if $s$ is a $K^\sharp$-moment sequence then $D(s)$ is a $K$-positivity preserver.
Conversely, by \Cref{cor:KwithCarleman} we have that if $D(s)$ is a determinate $K$-positivity preserver then $s$ is a determinate $K^\sharp$-moment sequence.
For special cases the reverse holds also without the determinacy assumption.

\begin{prop}\label{prop:KposPresCone}
Suppose that $K$ is a closed convex cone in $\rset^n$ with apex $0$, i.e., $K^\sharp = K$.
Then the following are equivalent:
\begin{enumerate}[(i)]
\item $D(s)$ is a $K$-positivity preserver.

\item $s$ is a $K^\sharp$-moment sequence.
\end{enumerate}
\end{prop}
\begin{proof}
Since $K$ is a closed convex cone with apex $0$ we have that $K^\sharp = K$ holds by Example \ref{exm:Ksharp} (c).

(i) $\Rightarrow$ (ii):
Let $y = 0$.
Then by \Cref{thm:KposDescription} we have that $s$ is a $K$-moment sequence.

(ii) $\Rightarrow$ (i): That is \Cref{lem:sKsharpMomSeq}.
\end{proof}

\begin{thm}\label{prop:KposPresStripe}
Let $C\subseteq\rset^n$ be compact and set $K := C\times [0,\infty)$, i.e., $K^\sharp = \{0\}\times [0,\infty)$.
Then the following are equivalent:
\begin{enumerate}[(i)]
\item $D(s)$ is a $K$-positivity preserver.

\item $s$ is a $K^\sharp$-moment sequence.
\end{enumerate}
\end{thm}
\begin{proof}
(i) $\Rightarrow$ (ii):
By \Cref{thm:KposDescription} we have that $s$ is a $(K-y)$-moment sequence for all $y\in K$.
Let $\mu_y$ be a representing measure of $s$.
Then $\mu_i$ defined by $\mu_i(A) := \mu(\rset^{i-1}\times A\times \rset^{n-i}\times [0,\infty))$ for any Borel set $A\subseteq\rset$ is a representing measure of $(s_{ke_i})_{k\in\nset_0}$ for $i=1,\dots,n$.
Since $D(s)$ is a $K$-positivity preserver we can restrict it to $\rset[x_i]$.
Then by \Cref{cor:compactK} we have that $\supp\mu_i = \{0\}$ for $i=1,\dots,n$ and hence $\supp\mu = \{0\}\times [0,\infty)$.
Since $\{0\}\times [0,\infty)$ is a closed convex cone with apex $0$ the assertion follows from \Cref{prop:KposPresCone}.

(ii) $\Rightarrow$ (i): This is \Cref{lem:sKsharpMomSeq}.
\end{proof}

\section{Generators of $K$-Positivity Preserving Semi-Groups}
\label{sec:generators}

In the previous section we have characterized all $K$-positvity preservers.
We want to deal now with $K$-positivity preserving semi-groups and their generators.
We start with some definitions.

\begin{dfn}
Let $\cC$ be a convex subset of $\rset[x_1,\dots,x_n]$ which is closed in the LF-topology of $\rset[x_1,\dots,x_n]$.
We denote by
\[\fD_\cC := \left\{ T\in\fD \,\middle|\, T\cC\subseteq\cC\right\}\]
the set of all \emph{$\cC$-preserves} and by
\[\fd_\cC := \left\{ A\in\fd \,\middle|\, e^{tA}\in\fD_\cC\ \text{for all}\ t\geq 0\right\}\]
the set of all \emph{generators of $\cC$-preserving semi-groups}.
If $\cC = \pos(\rset^n)$ we abbreviate
\[\fD_+ := \fD_{\pos(\rset^n)} \quad\text{and}\quad \fd_+ := \fd_{\pos(\rset^n)}.\]
\end{dfn}

Positive semi-groups on Banach lattices can be characterized in terms of their resolvents.
It seems that the case of the LF-space $\rset[x_1,\dots,x_n]$ is rarely considered in the literature; \cite{didio24posPresConst} might be the first work on this topic.
Following ideas from the proof of \cite[Prop.\ 2.3]{ouhabaz05} we obtain the following  description of $\fd_\cC$ via the resolvent.
Note that we only assume $\cC$ to be convex and closed; it does \emph{not} need to be a cone.

\begin{prop}\label{thm:fdCcara}
Let $\cV$ be a linear subspace of $\rset[x_1,\dots,x_n]$ and let $\cC$ be a non-empty convex subset of $\cV$ which is closed in the LF-topology.
Let $A\in\fd$ be such that $A\cV\subseteq\cV$.
For $d\in\nset_0$ we set
\[A_d := A|_{\cV\cap\rset[x_1,\dots,x_n]_{\leq d}} \qquad\text{and}\qquad \cC_d := \cC\cap\rset[x_1,\dots,x_n]_{\leq d}.\]
Then the following are equivalent:
\begin{enumerate}[(i)]
\item $A\in\fd_\cC$.

\item For every $d\in\nset_0$ there exists an $\varepsilon_d>0$ such that
\[(\one - \lambda A_d)^{-1}\cC_d \subseteq \cC_d\]
holds for all $\lambda\in [0,\varepsilon_d)$.
\end{enumerate}
\end{prop}
\begin{proof}
(i) $\Rightarrow$ (ii):
Let $d\in\nset_0$ and $p\in\cC_d$.
Let $\|\cdot\|_d$ be a norm on $\rset[x_1,\dots,x_n]_{\leq d}$.
We denote by $\|\cdot\|_d$ the corresponding operator norm for operators $\rset[x_1,\dots,x_n]_{\leq d}\to \rset[x_1,\dots,x_n]_{\leq d}$.
For $\gamma > \|A_d\|_d$ and all $t\geq 0$ we have 
\[ \left\| e^{-\gamma t}\cdot e^{tA_d}p \right\|_d \leq \|p\|_d\cdot e^{(\|A_d\|_d-\gamma)\cdot t}\]
and hence
\[ \gamma(\gamma\one - A_d)^{-1}\, p = \gamma\cdot\int_0^\infty e^{-\gamma t}\cdot e^{tA_d} p~\diff t .\]
Therefore, since $\gamma\int_0^\infty e^{-\gamma t}~\diff t = 1$ we conclude that $\lambda\cdot (\lambda\one - A_d)^{-1} p$ is a convex linear combination of elements $e^{tA_d}p$ in the closed convex set $\cC_d$, i.e.,
\[\gamma(\gamma\one - A_d)^{-1}\cC_d\subseteq\cC_d\]
for all $\gamma> \|A_d\|_d$.
With $\varepsilon_d := \|A_d\|_d^{-1}$ and hence $\lambda := \gamma^{-1} \in [0,\varepsilon_d)$ we get the assertion (ii) for every $d\in\nset_0$.

(ii) $\Rightarrow$ (i): Let $t\geq 0$ and $p\in\cC$.
Set $d := \deg p$.
Then we have that
\[e^{tA}p\quad  \overset{\text{Cor.\ \ref{cor:Ad}}}{=}\quad e^{tA_d}p
\quad \overset{\text{Cor.\ \ref{cor:expLimAd}}}{=}\quad \lim_{k\to\infty} \left(\one - \frac{tA_d}{k}\right)^{-k}p \quad\in\cC_d\]
holds since
\[\left(\one - \frac{tA_d}{k}\right)^{-1}\cC_d \subseteq\cC_d\]
holds for all $k > \varepsilon_d^{-1}$.
The limit $k\to\infty$ of
\[\left(\one - \frac{tA_d}{k}\right)^{-k}p\]
is in $\cC_d$ because $\cC_d$ is closed.
Since $t\geq 0$ and $p\in\cC$ were arbitrary we have that $e^{tA}\cC\subseteq\cC$ holds for all $t\geq 0$, i.e., $A\in\fd_\cC$.
\end{proof}

The linear subspace $\cV$ allows one to treat the sparse case when not all monomials of $\rset[x_1,\dots,x_n]$ are in $\cV$.
But in this general case the description (ii) via the resolvent might be hard to check even in the easiest cases.

\begin{exms}\label{exm:resolventProps}
Let $n = 1$, let $\cV=\rset[x]$, and let $\cC = \pos(\rset)$.
\begin{enumerate}[\bfseries\; (a)]
\item For $A = a\in\rset$ we have 
\[(\one - \lambda a)^{-1}>0\]
for all $\lambda\in (-\infty,a^{-1})$.
Then $A$ acts as the multiplication with a positive real constant, i.e., it preserves all cones $\cC_d$ with $\varepsilon_d = a^{-1}$.

\item For $A = \partial_x$ we have 
\[(\one - \lambda\partial_x)^{-1} = \sum_{k\in\nset_0} \lambda^k\cdot\partial_x^k \quad\in\fD_\cC\]
for all $\lambda\in\rset$ since $s_\lambda := (\lambda^k)_{k\in\nset_0}$ is the moment sequence with representing measure $\mu_\lambda := \delta_\lambda$, i.e., $\pm\partial_x\in\fd_\cC$.

\item For $A = \partial_x^2$ we have
\[(\one - \lambda\partial_x^2)^{-1} = \sum_{k=0}^\infty \lambda^k\cdot \partial_x^{2k} \quad\in \fD_\cC\]
for all $\lambda\in [0,\infty)$ since $s_\lambda = (1,0,\lambda,0,\lambda^2,0,\lambda^3,\dots)$ is a moment sequence represented by the measure $\mu_\lambda = \frac{1}{2}\delta_{-\sqrt{\lambda}} + \frac{1}{2}\delta_{\sqrt{\lambda}}$, i.e., $\partial_x^2\in\fd_\cC$.

\item Let $A = x\partial_x$.
Since $(e^{tA}p)(x) = p(e^t x)$ we have $\pm A\in\fd_\cC$.
Then $Ax^d = dx^d$ holds for all $d\in\nset_0$, i.e.,
\[(\one - \lambda x\partial_x)^{-1}x^d = (1-\lambda d)^{-1}\,x^d\]
holds for all $d\in\nset_0$ and $\lambda\in (-\infty,d^{-1})$ with $0^{-1} = +\infty$.
Here, $\varepsilon_d$ depends explicitly on the degree $d\in\nset_0$ of the restriction $\cC_d$.
\exmsymbol
\end{enumerate}
\end{exms}

\begin{rem}
In the preceding examples the following cases for $[0,\varepsilon_d)$ appeared:
\begin{enumerate}[\it\; (a)]
\item In Example \ref{exm:resolventProps} (a) the interval $[0,\varepsilon_d)$ is bounded from above and open at the right hand side.

\item In Example \ref{exm:resolventProps} (c) the interval $[0,\varepsilon_d)$ is  bounded from below by $0$ and it is closed at the left hand side at $0$.
To see this put $p(x) = x^2\in\pos(\rset)$.
Then
\[(\one-\lambda\partial_x^2)^{-1} \,x^2 = (1 +\lambda\partial_x^2)\, x^2 = x^2 +2\lambda\]
is negative at $x=0$ for any $\lambda<0$, i.e., $(\one-\lambda\partial_x^2)^{-1}x^2\not\in\pos(\rset)$ for any $\lambda<0$.

\item In Example \ref{exm:resolventProps} (d) we have $\varepsilon_d\rightarrow 0$ as $d\to\infty$ in general.
\end{enumerate}
Thus, in \Cref{thm:fdCcara} we can only have $\lambda\in [0,\varepsilon_d)$ with $\lim_{d\to \infty}\varepsilon_d= 0$.
\exmsymbol
\end{rem}

\Cref{cor:expLimAd} and the part ``(ii) $\Rightarrow$ (i)'' of the proof of \Cref{thm:fdCcara} yield the following.

\begin{prop}\label{prop:1plusAd}
Let $\cC$ be convex and closed in the LF-topology of $\rset[x_1,\dots,x_n]$ and let $A\in\fd$.
For every $d\in\nset_0$ we set
\[A_d := A|_{\rset[x_1,\dots,x_n]_{\leq d}} \qquad\text{and}\qquad \cC_d := \cC\cap\rset[x_1,\dots,x_n]_{\leq d}.\]
Assume that
\begin{enumerate}[(i)]
\item for every $d\in\nset_0$ there exists a $\varepsilon_d>0$ such that for all $\lambda\in [0,\varepsilon_d)$ we have
\[(\one + \lambda A_d)\cC_d\subseteq\cC_d.\]
\end{enumerate}
Then
\begin{enumerate}[(i)]\setcounter{enumi}{1}
\item $A\in\fd_\cC$.
\end{enumerate}
\end{prop}

In the case when  $\cC$ is a cone we have the following result.

\begin{cor}\label{cor:fDfdInclusion}
Let $\cC$ be a convex cone of $ \rset[x_1,\dots,x_n]$ with apex $0$ which is closed in the LF-topology of $\rset[x_1,\dots,x_n]$.
Then
\[\fD_\cC \;\subseteq\; \fd_\cC.\]
\end{cor}
\begin{proof}
Let $A\in\fD_\cC$, i.e., $A\cC_d\subseteq\cC_d$ and hence $(\one + \lambda A)\cC_d\subseteq\cC_d$ for all $d\in\nset_0$ and all $\lambda\in [0,\infty)$.
By \Cref{prop:1plusAd} we have $A\in\fd_\cC$ which proves $\fD_\cC\subseteq\fd_\cC$.
\end{proof}

\begin{rem}
The $\exp$-function in \Cref{cor:fDfdInclusion} can be replaced by a larger class of functions $f$ such that $f(\fD_\cC)\subseteq\fD_\cC$ holds.
Instead of $e^x = \sum_{k\in\nset_0} \frac{x^k}{k!}\in\rset[[x]]$ we can take more functions
\[f(x) = \sum_{k\in\nset_0} a_k x^k \quad\in\rset[[x]]\]
with $a_k\geq 0$ for all $k\in\nset_0$.
The coefficients $a_k$ have to be chosen such that $f(A) = \sum_{k\in\nset_0} a_k A^k$ converges in $\cset[x_1,\dots,x_n]^\sharp[[\partial_1,\dots,\partial_n]]$ for all $A\in\fD_\cC$, see \Cref{dfn:polySharpPartial} and \Cref{lem:polySharpPartialFrechetSpace}.
Then $f(\fD_\cC)\subseteq\fD_\cC$.
Such functions are e.g.\ absolutely monotonic functions with the appropriate conditions on the coefficients $a_k$.
Absolutely and completely monotonic functions are closely connected to the Laplace transform \cite{widder72} which was used in the proof of \Cref{thm:fdCcara} for the resolvent.
\exmsymbol
\end{rem}

The characterization of generators of $\cC$-preserving semi-groups in terms of resolvents given above is hard to check even in the case $\cC = \pos(\rset^n)$.
In \Cref{sec:KposPres} we have seen that the non-constant coefficient case can be reduced to the constant coefficient case and \Cref{thm:PosGeneratorsConstCoeffs} gives a description of the constant coefficient generators for $K=\rset^n$.

In \Cref{prop:KposPresStripe} we have given a description of $C\times [0,\infty)$-positivity preservers with constant coefficients where $C$ is a compact set.
Our next theorem contains a complete description of generators of $C\times [0,\infty)$-positivity preserver semi-groups.
For this we recall the L\'evy--Khinchin formula on $[0,\infty)$ for infinitely divisible measures on $[0,\infty)$, see e.g.\ \cite[Satz 16.14]{klenkewtheorie}.
It states that a probability measure $\mu$ on $[0,\infty)$ is infinitely divisible if and only if its $\log$-Laplace transform
\[\hat{\mu}(t) := -\log \int_0^\infty e^{-tx}~\diff\mu(x)\]
has the form
\[\hat{\mu}(t) = b t + \int_0^\infty (1-e^{-tx})~\diff\nu(x)\]
for all $t\geq 0$ with some number $b\geq 0$ and some Radon measure $\nu$ on $(0,\infty)$ such that
\[\int_0^\infty \max\{1,x\}~\diff\nu(x) < \infty.\]

\begin{thm}\label{thm:generators0infty}
Let $C\subseteq\rset^{n-1}$ be compact and let
\[A = \sum_{\alpha\in\nset_0^n} \frac{a_\alpha}{\alpha!}\cdot\partial^\alpha\]
with $a_\alpha\in\rset$ for all $\alpha\in\nset_0^n$.
Then the following are equivalent:
\begin{enumerate}[(i)]
\item $e^{tA}$ is a $C\times [0,\infty)$-positivity preserver for all $t\geq 0$.

\item $e^{A}$ has an infinitely divisible representing measure supported on $\{0\}\times(0,\infty)$.

\item $e^{tA}$ has an infinitely divisible representing measure supported on $\{0\}\times(0,\infty)$ for all $t\geq 0$.

\item $e^{tA}$ has an infinitely divisible representing measure supported on $\{0\}\times(0,\infty)$ for one $t>0$.

\item There exist a measure Radon $\nu$ on $(0,\infty)$ satisfying
\[\int_0^\infty x^k~\diff\nu(x) < \infty\]
for all $k\in\nset_0$ and a constant $b\geq 0$ such that
\[a_0\in\rset,\qquad a_{e_n} = b + \int_0^\infty x~\diff\nu(x),\qquad \text{and}\qquad a_{k\cdot e_n} = \int_0^\infty x^k~\diff\nu(x)\]
hold for all $k\in\nset$ with $k\geq 2$ and $a_\alpha=0$ holds otherwise.
\end{enumerate}
\end{thm}
\begin{proof}
We adapt the idea of the proof of \cite[Thm.\ 4.7 and 4.11]{didio24posPresConst}.

Without loss of generality we set $a_0 = 0$ since a scaling by $e^{ta_0}$ preserves positivity and commutes with all $\partial^\alpha$.
Further, by \Cref{prop:KposPresStripe} we have that all $C\times [0,\infty)$-positivity preservers are supported on $\{0\}\times [0,\infty)$.
Hence, it is sufficient to prove the result for $n=1$.

(i) $\Rightarrow$ (ii):
Since $e^{tA}$ is a $[0,\infty)$-positivity preserver for all $t\geq 0$ there exists a representing measure $\mu_t$ on $[0,\infty)$ by \Cref{prop:KposPresStripe}.
Set $\nu_k := (\mu_{1/k!})^{*k!}$ for all $k\in\nset$.
Then $\nu_k$ is a representing measure of $e^{A}$ for all $k\in\nset$.
Since $\rset[x]$ is an adapted space we have that $(\nu_k)_{k\in\nset}$ is vaguely compact by \cite[Thm.\ 1.19]{schmudMomentBook} and there exists a subsequence $(k_i)_{i\in\nset}$ such that $\nu_k\to\nu$ and $\nu$ is a representing measure of $e^A$.
All measures $\nu_k$, $k\in\nset$, and $\nu$ are supported on $[0,\infty)$.

We have to show that $\nu$ is infinitely divisible, i.e., for every $l\in\nset$ there exists a Radon measure $\omega_l$ on $[0,\infty)$ with $\omega_l^{*l} = \nu$.

Let $l\in\nset$.
For $i\geq l$ we define
\[\omega_{l,i}:= (\mu_{1/k_1!})^{*k_i!/l},\]
i.e., $\omega_{l,i}$ is a representing measure of $e^{A/l}$.
Again, since $\rset[x]$ is an adapted space by \cite[Thm.\ 1.19]{schmudMomentBook} there exists a subsequence $(i_j)_{j\in\nset}$ such that $\omega_{l,i_j}$ converges to some $\omega_l$.
Hence, 
\[(\omega_l)^{*l} = \lim_{j\to\infty} (\omega_{l,i_j})^{*l} = \lim_{j\to\infty} \nu_{k_{i_j}} = \nu,\]
i.e., $\nu$ is divisible by all $l\in\nset$.
Therefore, we have that $\nu$ is an infinitely divisible representing measure of $e^A$.

(ii) $\Rightarrow$ (i):
Let $\mu_1$ be an infinitely divisible representing measure of $e^A$.
Then $\mu_q := \mu_1^{*q}$ exists for all $q\in\qset\cap [0,\infty)$.
Since $\exp:\fd\to\fD$ is continuous by \Cref{thm:DregularFrechetLieGroup} and the set of $K$-positivity preservers is closed by \Cref{lem:closedPosPres} it follows that $e^{qA}$ is a $[0,\infty)$-positivity preserver for all $q\geq 0$.
This proves (i).

The implications ``(i) $\Leftrightarrow$ (iii)'' and ``(i) $\Leftrightarrow$ (iv)'' follow from ``(i) $\Leftrightarrow$ (ii)'' by setting $\tilde{A} = tA$ for $t>0$.

(ii) $\Leftrightarrow$ (v): 
By (ii) there exists an infinitely divisible representing measure $\mu$ of $e^A$.
Then we have
\begin{equation}\label{eq:proofHalf1}
e^A = \sum_{k\in\nset_0} \frac{1}{k!}\int_0^\infty x^k~\diff\mu(x)\cdot \partial_x^k.
\end{equation}
By \cite[Thm.\ 2.17]{didio24posPresConst} we can take the logarithm since we are in the constant coefficients case.
Hence, (\ref{eq:proofHalf1}) is equivalent to
\begin{equation}\label{eq:proofHalf2}
A = \sum_{k\in\nset} \frac{a_k}{k!}\cdot\partial_x^k = \log\left( \sum_{k\in\nset_0} \frac{1}{k!}\cdot \int_0^\infty x^k~\diff\mu(x)\cdot\partial_x^k \right).
\end{equation}
With the isomorphism
\[\cset[[\partial_x]] \to \cset[[t]],\quad \partial_x\mapsto -t\]
it follows that (\ref{eq:proofHalf2}) is equivalent to
\begin{align}
\sum_{k\in\nset} \frac{a_k}{k!}\cdot (-t)^k &= \log \left( \sum_{k\in\nset_0} \frac{1}{k!}\cdot \int_0^\infty x^k~\diff\mu(x) \cdot (-t)^k \right).\label{eq:proofHalf3}
\intertext{The right hand side of (\ref{eq:proofHalf3}) is}
&= \log \int_0^\infty e^{-tx}~\diff\mu(x),\notag
\intertext{i.e., it is the negative $-\hat{\mu}$ of the characteristic function $\hat{\mu}$ of the measure $\mu$.
By the L\'evy--Khinchin formula on $[0,\infty)$, see \ \cite[Satz 16.14]{klenkewtheorie}, we obtain}
&= -bt + \int_0^\infty (e^{-tx}-1)~\diff\nu(x).\label{eq:proofHalf4}
\intertext{After a formal power series expansion of $e^{-tx}$ in the Fr\'echet topology of $\rset[[t]]$ we have that (\ref{eq:proofHalf4}) is equivalent to}
\sum_{k\in\nset} \frac{a_k}{k!}\cdot (-t)^k &= -bt + \sum_{k\in\nset} \frac{1}{k!}\cdot \int_0^\infty x^k~\diff\nu(x)\cdot (-t)^k.\label{eq:proofHalf5}
\end{align}
A comparison of the coefficients in $\rset[[t]]$ gives (\ref{eq:proofHalf5}) $\Leftrightarrow$ (v).
\end{proof}

\begin{rem}
By a slight modification of the preceding proof it follows that a similar result as \Cref{thm:PosGeneratorsConstCoeffs} holds for sets $K = C\times \rset^n$ with $C\subseteq\rset^m$ compact, $n,m\in\nset_0$.
The coefficients $a_{(\alpha,\beta)}$ with $(\alpha,\beta)\in\nset_0^m\times\nset_0^n$ are then the coefficients from \Cref{thm:PosGeneratorsConstCoeffs} (ii) for all $a_{(0,\beta)}$ with $\beta\in\nset_0^n$ and $a_{(\alpha,\beta)}=0$ for all $\alpha\in\nset_0^m\setminus\{0\}$ and $\beta\in\nset_0^n$.
\exmsymbol
\end{rem}

By \Cref{thm:generators0infty} we see that the only generators of finite degree of $[0,\infty)$-positivity preserving semi-groups with constant coefficients are of the form
\[a_0 + a_1\partial_x\]
with $a_0\in\rset$ and $a_1 \geq 0$, i.e., the semi-group acts by scaling and translation.

Let us compare the generators with constant coefficients on $\rset$, see \Cref{thm:PosGeneratorsConstCoeffs}, and the generators with constant coefficients on $[0,\infty)$, see \Cref{thm:generators0infty}.
On $[0,\infty)$ the contributions of the second derivatives $\sigma_{i,j}\partial_i \partial_j$ are missing since they correspond to the solution of the heat equation.
The convolution with the heat kernel does not preserve $[0,\infty)$.
On $\rset$ the measure $\nu$ can have a pole at $x=0$, i.e.,
$\int_\rset |x|~\diff\nu(x)$ is not necessarily finite, so $\nu$ it is not a moment measure in general.
On $[0,\infty)$ we have that $\nu$ is a moment measure on $[0,\infty)$ such that $\nu(\{0\})=0$.

\begin{cor}
Let
\[A = \sum_{k\in\nset_0} \frac{a_k}{k!}\cdot\partial_x^k\]
with $a_k\in\rset$ for all $k\in\nset_0$.
Then the following are equivalent:
\begin{enumerate}[(i)]
\item $e^{tA}$ is a $[0,\infty)$-positivity preserver for all $t\geq 0$.

\item There exist a $[0,\infty)$-moment sequence $s=(s_k)_{k\in\nset_0}$ and constants $\alpha\in\rset$ and $\beta\geq 0$ such that
\[a_0 = s_0 + \alpha,\qquad a_1 = s_1 + \beta,\qquad \text{and}\qquad a_k=s_k\ \text{for all}\ k\geq 2.\]
\end{enumerate}
If $\mu$ is a representing measure of $s$ then
\[e^{\alpha t}\cdot e^{*t\mu}*\delta_{\beta t}\]
is a representing measure of $e^{tA}$ for all $t\geq 0$ with
\[e^{*t\mu} = \sum_{k\in\nset_0} \frac{t^k}{k!}\cdot \mu^{*k}.\]
\end{cor}
\begin{proof}
The equivalence ``(i) $\Leftrightarrow$ (ii)'' is \Cref{thm:generators0infty} ``(i) $\Leftrightarrow$ (v)''.

Let $\mu$ be a representing measure of the sequence $s = (s_k)_{k\in\nset_0}$ and set $a := (\alpha\cdot\delta_{0,k})_{k\in\nset_0} = (\alpha,0,0,\dots)$ and $b := (\beta\cdot\delta_{1,k})_{k\in\nset_0} = (0,b,0,0,\dots)$.
Then we have $A = D(s) + D(a) + D(b)$.
Since these differential operators have constant coefficients all three operators commute.
Hence,
\[e^{tA} = e^{tD(a)} e^{tD(s)} e^{tD(b)}\]
holds for all $t\geq 0$.
From
\begin{align*}
e^{tD(a)} &= e^{\alpha t},\\
e^{tD(s)} &= \sum_{k\in\nset_0} \frac{t^k}{k!}\cdot D(s)^k = \sum_{k\in\nset_0} \frac{t^k}{k!}\cdot D(s^{*k}),
\intertext{and}
e^{tD(b)} &= e^{\beta t\partial_x} = \sum_{k\in\nset_0} \frac{(\beta t)^k}{k!}\cdot \partial_x^k
\end{align*}
it follows that $e^{tD(a)}$ is represented by the measure $e^{\alpha t}\cdot\delta_0$, $e^{tD(s)}$ is represented by  $e^{*t\mu}$, and $e^{tD(b)}$ as the shift operator is represented by $\delta_{\beta t}$ for all $t\geq 0$.
\end{proof}

It is clear that if $\mu$ is a moment measure then $e^{*t\mu}$ is an infinitely divisible moment measure for all $t\geq 0$.
If $\mu$ is indeterminate then $e^{*\mu}$ is indeterminate.

We want to use the description of generators of $K$-positive semi-groups with constant coefficients for the non-constant coefficient case.
The following direction always holds.

\begin{thm}\label{prop:genK}
Let $K\subseteq\rset^n$ be closed and let
\[A = \sum_{\alpha\in\nset_0^n} \frac{a_\alpha}{\alpha!}\cdot \partial^\alpha\]
be with $a_\alpha\in\rset[x_1,\dots,x_n]_{\leq |\alpha|}$ for all $\alpha\in\nset_0^n$.
If
\begin{enumerate}[(i)]
\item for all $y\in K$ the operator
\[A_y := \sum_{\alpha\in\nset_0^n} \frac{a_\alpha(y)}{\alpha!}\cdot\partial^\alpha\]
is a generator of a $K$-positivity preserving semi-group
\end{enumerate}
then
\begin{enumerate}[(i)]\setcounter{enumi}{1}
\item $e^{tA}$ is a $K$-positivity preserver for all $t\geq 0$.
\end{enumerate}
\end{thm}
\begin{proof}
We introduce $y=(y_1,\dots,y_n)\in K$ as new variables.
Then
\[e^{tA_y} = \sum_{k\in\nset_0} \frac{t^k}{k!} A_y^k\]
for all $t\in\rset$ by \Cref{thm:DregularFrechetLieGroup}, i.e., we have $e^{tA_y} = \one + t A_y + t^2 R_y(t)$ for some error term $R_y(t)$ such that $R_y(t)$ on $\rset[x_1,\dots,x_n]_{\leq d}$ is uniformly bounded for all $t\in [0,1]$ and for fixed $d\in\nset_0$.

Let $f\in\rset[x_1,\dots,x_n]$.
Since $\deg a_\alpha\leq |\alpha|$ for all $\alpha\in\nset_0^n$ we have
\[g(x,y,t,f) := e^{tA_y}f = (\one + tA_y + t^2 R_y(t))f\in \rset[x_1,\dots,x_n,y_1,\dots,y_n]_{\leq \deg f}\]
for all $t\in\rset$.
By (i) we have that $A_y$ is a generator of a $K$-positivity preserver with constant coefficients, i.e.,
\[g(x,y,t,f)\geq 0\]
for all $f\in\pos(K)$, $x,y\in K$, and $t\geq 0$.
Hence, with $x=y$ we obtain
\begin{equation}\label{eq:gpos}
g(x,x,t,f) \geq 0
\end{equation}
for all $f\in\pos(K)$, $x\in K$, and $t\geq 0$.
For $t\in\rset$ we define
\[G_t:\rset[x_1,\dots,x_n]\to\rset[x_1,\dots,x_n]\]
by
\[(G_t f)(x) := (e^{tA_y}f)(x,x) = g(x,x,t,f),\]
i.e., we have
\[G_t f = (\one + tA + t^2 R(t)) f\]
for some error term $R(t)$ which is on $\rset[x_1,\dots,x_n]_{\leq d}$, $d\in\nset_0$, uniformly bounded in $t\in [0,1]$.
By (\ref{eq:gpos}) we have that $G_t$ is a $K$-positivity preserver for all $t\geq 0$.
Hence,
\[e^{tA} \overset{\text{Cor.\ \ref{cor:expLimAd}}}{=} \lim_{k\to\infty} \left(\one + \frac{t A}{k} \right)^k = \lim_{k\to\infty} \left(\one + \frac{t A + t^2 R(t)}{k} \right)^k = \lim_{k\to\infty} G_{t/k}^k\]
and since all $G_{t/k}$ are $K$-positivity preservers also $G_{t/k}^k$ are $K$-positivity preservers as well.
Therefore, since the set of $K$-positivity preservers is closed by \Cref{lem:closedPosPres} (ii) we have that $e^{tA}$ is a $K$-positivity preserver.
Since $t\geq 0$ was arbitrary we have that (ii) is proved.
\end{proof}

The implication ``(ii) $\Rightarrow$ (i)'' in \Cref{prop:genK} does not hold in general, see e.g.\ \Cref{rem:shift}.
For $K=\rset^n$ the implication holds and we get the following.

\begin{thm}\label{thm:PosGeneratorsGeneral}
Let
\[A = \sum_{\alpha\in\nset_0^n} \frac{a_\alpha}{\alpha!}\cdot\partial^\alpha\]
be with $a_\alpha\in\rset[x_1,\dots,x_n]_{\leq |\alpha|}$ for all $\alpha\in\nset_0^n$.
Then the following are equivalent:
\begin{enumerate}[(i)]
\item $A\in\fd_+$, i.e., $e^{tA}$ is a positivity preserver for all $t\geq 0$.

\item For every $y\in\rset^n$ there exist a symmetric matrix $\Sigma(y) = (\sigma_{i,j}(y))_{i,j=1}^n$ with real entries such that $\Sigma(y)\succeq 0$, a vector $b(y) = (b_1(y),\dots,b_n(y))^T\in\rset^n$, $a_0\in\rset$, and a measure $\nu_y$ on $\rset^n$ with
\[\int_{\|x\|_2\geq 1} |x_i|~\diff\nu_y(x)<\infty \qquad\text{and}\qquad \int_{\rset^n} |x^\alpha|~\diff\nu_y(x)<\infty\]
for $i=1,\dots,n$ and $\alpha\in\nset_0^n$ with $|\alpha|\geq 2$ such that
\begin{align*}
a_{e_i}(y) &= b_i(y) + \int_{\|x\|_2\geq 1} x_i~\diff\nu_y(x) && \text{for}\ i=1,\dots,n,\\
a_{e_i+e_j}(y) &= \sigma_{i,j}(y) + \int_{\rset^n} x^{e_i + e_j}~\diff\nu_y(x) && \text{for all}\ i,j=1,\dots,n,
\intertext{and}
a_\alpha(y) &= \int_{\rset^n} x^\alpha~\diff\nu_y(x) &&\text{for}\ \alpha\in\nset_0^n\ \text{with}\ |\alpha|\geq 3.
\end{align*}
\end{enumerate}
\end{thm}
\begin{proof}
By \Cref{thm:PosGeneratorsConstCoeffs} we have that (i) is equivalent to the fact that
\[A_y = \sum_{\alpha\in\nset_0^n} \frac{a_\alpha(y)}{\alpha!}\cdot\partial^\alpha\]
is a generator of a positivity preservering semi-group with constant coefficients for all $y\in\rset^n$.

(ii) $\Rightarrow$ (i):
That is \Cref{prop:genK}.

(i) $\Rightarrow$ (ii):
Let $y\in\rset^n$.
By (i) we have $A\in\fd_+$, i.e., $T_t := e^{tA}$ is a positivity preserver for all $t\geq 0$.
Hence,
\begin{equation}\label{eq:yPos}
(T_{t,y} f)(y) = (T_t f)(y)\geq 0
\end{equation}
for $t\geq 0$ and $f\in\pos(\rset^n)$, i.e., the linear map $T_{t,y}$ with constant coefficients preserves positivity of $f\in\pos(\rset^n)$ at $x = y$.

Since $T_{t,y}$ is a linear operator with constant coefficients it commutes with $\partial_i$ for $i=1,\dots,n$ and hence with $e^{z\cdot\nabla}$ which is the shift $(e^{z\cdot\nabla}f)(y) = f(y+z)$ for any $z\in\rset^n$.
Hence, for $x\in\rset^n$ we have
\[(T_{t,y}f)(x) = (T_{t,y}f)(y+x-y) = (T_{t,y} e^{(x-y)\cdot\nabla} f)(y) = e^{(x-y)\cdot\nabla} (T_{t,y} f)(y) \geq 0,\]
by (\ref{eq:yPos}), i.e., $T_{t,y}$ is a positivity preserver.

By \Cref{thm:DregularFrechetLieGroup} we have
\[T_t = e^{tA} = \sum_{k\in\nset_0} \frac{t^k A^k}{k!} = \one + tA + t^2 R(t)\]
with some error term $R(t)$ which is uniformly bounded for all $t\in [0,1]$ on $\rset[x_1,\dots,x_n]_{\leq d}$ for all fixed $d\in\nset_0$. 
Hence, 
\[e^{tA_y} \overset{\text{Cor.\ \ref{cor:expLimAd}}}{=} \lim_{k\to\infty} \left(\one + \frac{t A_y}{k} \right)^k = \lim_{k\to\infty} \left(\one + \frac{t A_y + t^2 R_y(t)}{k} \right)^k = \lim_{k\to\infty} T_{t/k,y}^k.\]
Since $T_{t/k,y}$ is a positivity preserver so is $T_{t/k,y}^k$ and therefore by \Cref{lem:closedPosPres} (ii) we have that $e^{tA_y}$ is also a positivity preserver for $t\geq 0$.
In summary, $A_y$ is a generator of a positivity preservering semi-group with constant coefficients.
Since $y\in\rset^n$ was arbitrary we have proved (ii) by \Cref{thm:PosGeneratorsConstCoeffs}.
\end{proof}

For the Hilbert space case $L^2(\rset^n)$ it is shown in \cite{miyaji86} that the generators of finite orders of positive semi-groups are of order at most $2$ and are (degenerate) elliptic operators.
We see from \Cref{thm:PosGeneratorsGeneral} that this holds also for the generators of $\rset^n$-positivity preservering semi-groups, i.e., for $k\in\nset_0$ and
\[\sum_{\alpha\in\nset_0^n: |\alpha|\leq k} \frac{q_\alpha}{\alpha!} \cdot\partial^\alpha \quad\in\fd_+\]
we have that $k\leq 2$, $q_{e_i+e_j}\in\rset[x_1,\dots,x_n]_{\leq 2}$ for all $i,j=1,\dots,n$, and
\[\big( q_{e_i+e_j}(y) )_{i,j=1}^n \succeq 0\]
hold for all $y\in\rset^n$.
So for $n=1$ we have e.g.\ that
\[a_0\; +\; (b_0 + b_1x)\cdot\partial_x\; +\; \left(c_0^2 + (c_1+c_2 x)^2\right)\cdot\partial^2_x\]
with $a_0,b_0,b_1,c_0,c_1,c_2\in\rset$ are the only generators of positivity preservers on $\rset$ of finite degree.

The infinite tail from the (L\'evy) measures $\nu_y$ is specific to the polynomials.

In \Cref{thm:PosGeneratorsGeneral} we have shown that if $A$ is a generator of a $\rset^n$-positivity preserving semi-group with possibly non-constant coefficients then for every $y\in\rset^n$ the operator $A_y$ is a generator of a $\rset^n$-positivity preserving semi-group with constant coefficients.
On $[0,\infty)$ this is no longer true.

\begin{rem}\label{rem:shift}
Let $A = a x\partial_x$.
Then $A^k x^m = (am)^k x^m$ and hence
\[e^{tA} x^m = \sum_{k\in\nset_0} \frac{t^k A^k}{k!}x^m = \sum_{k\in\nset_0} \frac{(atm)^k}{k!} x^m = e^{atm} x^m = (e^{at} x)^m,\]
i.e., $(e^{tA}f)(x) = f(e^{at} x)$ is the scaling which is a $[0,\infty)$-positivity preserver for all $a\in\rset$.
But for $a=-1$ and $y=1$ we have $A_1 = -\partial_x$ which is not a generator of a $[0,\infty)$-positivity preserving semi-group with constant coefficients.
That is, we do not get a description of all generators because the operators $ax\partial_x$ with $a<0$ are not covered.
The reason for the failure is that in the proof of \Cref{thm:PosGeneratorsGeneral} the positivity at arbitrary $x\in\rset^n$ was derived from the positivity at $x = y$ by applying the shift $e^{c\partial_x}$ which commutes with $e^{A_y}$.
On $[0,\infty)$ this reasoning fails because $[0,\infty)$ is not translation invariant.
\exmsymbol
\end{rem}

\begin{rem}
Another example where ``(ii) $\Rightarrow$ (i)'' in \Cref{prop:genK} fails is when $K$ is the unit ball of $\rset^n$.
Then by \Cref{cor:compactK} the only positivity preservers with constant coefficients are $T = c\one\geq 0$.
Hence, rotations are not covered.
\exmsymbol
\end{rem}

\section{Eventually Positive Semi-Groups}
\label{sec:eventually}

Since we characterized the generators of positive semi-groups for $K=\rset^n$ we want to look at some operators $A$ which seem to generate positive semi-groups.

Semi-groups $(e^{tA})_{t\geq 0}$ such that $e^{tA}$ is positive for all $t\geq \tau$ for some $\tau>0$ but not for $t\in (0,\tau)$ are called (uniformly) \emph{eventually positive} (or non-negative) semi-groups \cite{daners18,glueck23}.

\begin{prop}\label{thm:sigma}
Let
\[T_t := e^{t\cdot (x\partial_x)^3}: \rset[x]_{\leq 4}\to\rset[x]_{\leq 4},\]
i.e.,
\[a_0 + a_1x + a_2 x^2 + a_3x^3 + a_4 x^4 \;\mapsto\; a_0 + a_1 e^t x + a_2 e^{8t} x^2 + a_3 e^{27t} x^3 + a_4 e^{64t} x^4.\]
Then there exists a constant $\tau\in\rset$ with
\[1.19688\cdot 10^{-2} ~~ <~~ \tau ~~<~~ 1.19689\cdot 10^{-2}\]
such that
\begin{enumerate}[(i)]
\item $T_t$ is not a positivity preserver for any $t \in (0,\tau)$
\end{enumerate}
and
\begin{enumerate}[(i)]\setcounter{enumi}{1}
\item $T_t$ is a positivity preserver for $t = 0$ and all $t\geq \tau$,
\end{enumerate}
i.e.,
\[\left(e^{t\cdot (x\partial_x)^3}\right)_{t\geq 0} \quad\text{on}\quad \rset[x]_{\leq 4}\]
is an \emph{eventually positive} semi-group.
\end{prop}
\begin{proof}
We have
$\lambda_k(t) = e^{t\cdot k^3}$
for all $k=0,\dots,4$.
Hence, we have
\[\cH(\lambda(t))_0 = (1)\succeq 0\]
since $\det\cH(\lambda(t))_0 = 1 \geq 0$ and
\[\cH(\lambda(t))_1 = \begin{pmatrix} 1 & e^t\\ e^t & e^{8t} \end{pmatrix}\succeq 0\]
since $\det\cH(\lambda(t))_1 = e^{8t} - e^{2t} = e^{2t}\cdot (e^{6t} - 1) \geq 0$ for all $t\geq 0$ with $\det\cH(\lambda(t))_1>0$ for $t>0$.
It is therefore sufficient to look at $\det\cH(\lambda(t))_2$ which is
\[h_2(t) := \det\cH(\lambda(t))_2 = e^{72t} - e^{66t} - e^{54t} + 2e^{36t} - e^{24t}.\]
From $h_2(0) = 0$, $h_2'(0) = 0$, and $h_2''(0) = -72$ we get $h(t) < 0$ for $t\in (0,\varepsilon)$ for some $\varepsilon > 0$.
By calculations from Mathematica \cite{mathematica} we get
\[\sigma_3(0.0119688) \approx -3.39928\cdot 10^{-8} \quad\text{and}\quad \sigma_3(0.0119689) \approx 1.7888\cdot 10^{-8},\]
i.e., $\cH(\lambda(t))_2$ becomes positive semi-definite between $t = 0.0119688$ and $t = 0.0119689$.
The smallest eigenvalues $\sigma_3(t)$ of $\cH(\lambda(t))_2$ are depicted in \Cref{fig:sigma} for $t\in [0,0.015]$, $t\in [0,1]$, and $t\in [0,10]$.
We have $h_2(t)<0$ for $t\in (0,\tau)$ with $\tau\in [0.0119688,0.0119689]$ which proves (i).
Since we know $\sigma_3(t)$ for $t\in [0,1]$, see \Cref{fig:sigma} (b), and since $h_3'(t) > 0$ holds for all $t\geq 1$ we have (ii).
\end{proof}

\begin{figure}[htb!]
\begin{subfigure}{.33\textwidth}
\includegraphics[width=\linewidth]{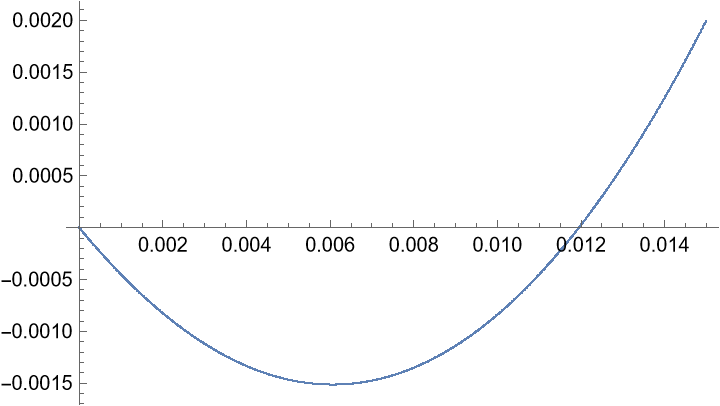}
\caption{$t\in [0,0.015]$}
\end{subfigure}
\begin{subfigure}{.325\textwidth}
\includegraphics[width=\linewidth]{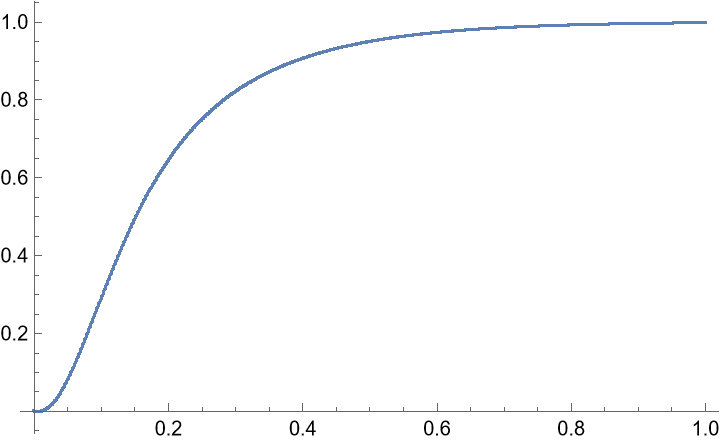}
\caption{$t\in [0,1]$}
\end{subfigure}
%
\begin{subfigure}{.33\textwidth}
\includegraphics[width=\linewidth]{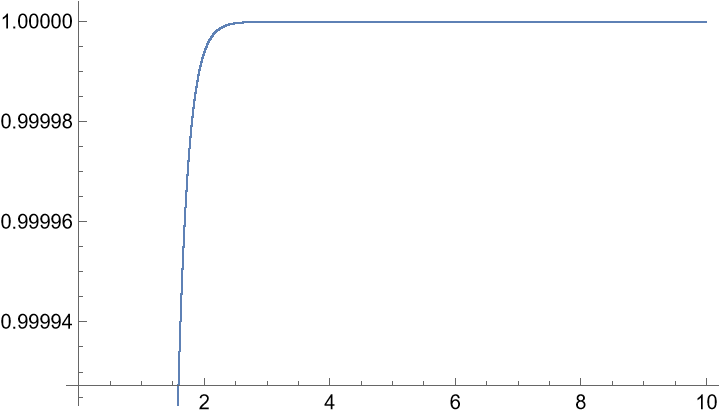}
\caption{$t\in [0,10]$}
\end{subfigure}
%
\caption{The smallest eigenvalue $\sigma_3(t)$ from \Cref{thm:sigma} for $t\in [0,10]$.\label{fig:sigma}}
\end{figure}

\begin{rem}
Note, that $T_t = e^{t (x\partial_x)^3}$ on $\rset[x]$, not just on $\rset[x]_{\leq 4}$, fulfills the following.
For any $p\in\pos(\rset)$ there exists a $\tau_p>0$ such that $T_tp \in\pos(\rset)$ holds for all $t\geq \tau_p$.
\exmsymbol
\end{rem}

For eventually positive semi-groups see e.g.\ \cite{daners18,denkploss21,glueck23}.
Our example $A = (x\partial_x)^3$ on $\rset[x]_{\leq 4}$ in \Cref{thm:sigma} acts only on the $5$-dimensional (Hilbert) space $\rset[x]_{\leq 4}$ and is therefore a matrix example of an eventually positive semi-group, see \cite{daners16b} and literature therein for the matrix cases.
Note, that in the matrix case the invariant cone is (usually) $[0,\infty)^n$.
In our example it is $\pos(\rset)_{\leq 4}$.
So there are slight differences.

Eventually positive semi-groups are an active field with the need for small examples which can be calculated and investigated \cite{glueck23}.
Therefore, we provide the smallest possible case of an eventually positive semi-group in our framework: A degree two operator on $\rset[x]_{\leq 2}$.

\begin{prop}\label{thm:1}
Let $a\in\rset$ be real with $|a|>\frac{1}{\sqrt{5}}$ and let
\[A = a\cdot\partial + \frac{1}{2}(x^2-1)\cdot\partial^2\]
be an operator on the $3$-dimensional Hilbert space $\cH = \rset[x]_{\leq 2}$.
Then
\[\left(e^{tA} \right)_{t\geq 0}\]
is an eventually positive semi-group on $\cH$, i.e., there exists a constant $\tau_a >0$ such that
\begin{enumerate}[(i)]
\item $e^{tA}\pos(\rset)_{\leq 2}\not\subseteq\pos(\rset)_{\leq 2}$ holds for all $t\in (0,\tau_a)$
\end{enumerate}
and
\begin{enumerate}[(i)]\setcounter{enumi}{1}
\item $e^{tA}\pos(\rset)_{\leq 2}\subseteq\pos(\rset)_{\leq 2}$ holds for all $t = 0$ and $t\geq \tau_a$.
\end{enumerate}
For $a\in\rset$ with $|a|\leq \frac{1}{\sqrt{5}}$ we have $e^{tA}\pos(\rset)_{\leq 2}\not\subseteq\pos(\rset)_{\leq 2}$ for all $t\in (0,\infty)$.
\end{prop}
\begin{proof}
Let $\{1,x,x^2\}$ be the monomial basis of $\cH$.
From
\[A1 = 0,\quad Ax = a,\quad \text{and}\quad Ax^2 = -1 + 2ax + x^2\]
we get the matrix representation
\[\tilde{A} = \begin{pmatrix} 0 & a & -1\\ 0 & 0 & 2a\\ 0 & 0 & 1 \end{pmatrix}\]
of $A$ on $\cH$ in the monomial basis.
Therefore, we get
\[\tilde{B} = \exp(t\tilde{A}) = \begin{pmatrix}
1 & at & (2a^2-1)\cdot (e^t - 1) - 2a^2 t\\
0 & 1 & 2a\cdot (e^t - 1)\\
0 & 0 & e^t
\end{pmatrix} = \begin{pmatrix}
1 & at & f(a,t)\\
0 & 1 & g(a,t)\\
0 & 0 & e^t
\end{pmatrix}\]
for all $t\in\rset$, e.g.\ by using \cite{mathematica}.
The matrix representation $\tilde{B}$ belongs to
\[B = b_0 + b_1\cdot\partial + \frac{b_2}{2}\cdot\partial^2\]
with
\begin{align*}
B1 = 1 = b_0 \quad &\Rightarrow\quad b_0 = 1,\\
Bx = at + x = b_0 x + b_1 \quad &\Rightarrow\quad b_1 = at
\end{align*}
and
\[Bx^2 = f(a,t) + g(a,t)\cdot x + e^t\cdot x^2 = b_0 x^2 + 2b_1 x + b_2\]
gives
\begin{align*}
b_2 &= f(a,t) + g(a,t)\cdot x + e^t\cdot x^2 - b_0 x^2 - 2b_1 x\\
&= (2a^2 - 1)\cdot (e^t-1) - 2a^2 t + 2a\cdot (e^t-1 - t)\cdot x + (e^t-1)\cdot x^2.
\end{align*}
By \cite[Thm.\ 3.1]{borcea11} we have that
\begin{equation}
B\pos(\rset)_{\leq 2}\subseteq\pos(\rset)_{\leq 2}
\end{equation}
is equivalent to
\begin{equation}\label{eq:3}
\begin{pmatrix}
b_0(x) & b_1(x)\\ b_1(x) & b_2(x)
\end{pmatrix}\succeq 0\ \text{for all}\ x\in\rset.
\end{equation}
Since $b_0 = 1>0$ we have that (\ref{eq:3}) is equivalent to
\begin{equation}\label{eq:4}
h(x) := \det\begin{pmatrix}
b_0(x) & b_1(x)\\ b_1(x) & b_2(x)
\end{pmatrix} = b_0(x)\cdot b_2(x) - b_1(x)^2 \geq 0\ \text{for all}\ x\in\rset
\end{equation}
with
\[h(x) = (2a^2 -1)\cdot (e^t - 1) - 2a^2 t - a^2 t^2 + 2a\cdot (e^t-1-t)\cdot x + (e^t-1)\cdot x^2.\]
For $t>0$ we have that $h$ is a quadratic polynomial in $x$ and hence we have that (\ref{eq:4}) is equivalent to
\begin{equation}
m := \min_{x\in\rset} h(x) \geq 0
\end{equation}
with
\begin{equation}\label{eq:min}
m = 1-e^t + a^2\cdot\frac{5 + 8t + 4t^2 - (10+8t+t^2)\cdot e^t + 5\cdot e^{2t}}{e^t-1}.
\end{equation}
Hence, we have that for each $a\in\rset$ with $|a|> \frac{1}{\sqrt{5}}$ that there exists a $\tau_a>0$ such that $m(a,t)<0$ for all $t\in (0,\tau_a)$ and $m(a,t) > 0$ for all $t\in (\tau_a,\infty)$ hold, i.e., (i) and (ii) are proved.

Let $a\in\rset$ be with $|a|<\frac{1}{\sqrt{5}}$.
We see that $m(a,t)<0$ for all $t>0$ and for the boundary case $a=\pm\frac{1}{\sqrt{5}}$ we have
\[m(\pm 5^{-1/2},t) = \frac{4 e^{-t}\cdot t\cdot (t+2) - t\cdot (t+8)}{5 - 5e^{-t}} < 0\]
also for all $t>0$ which proves $e^{tA}\pos(\rset)_{\leq 2}\not\subseteq\pos(\rset)_{\leq 2}$ for all $t>0$.
\end{proof}

It is clear from (\ref{eq:min}) that we have $\tau_a = \tau_{-a}$ for all $a\in\rset$ with $|a|>\frac{1}{\sqrt{5}}$.
In \Cref{fig:2} we give the minima $m(a,t)$ from (\ref{eq:min}) for two values $a$ close to $\frac{1}{\sqrt{5}}$.
\begin{figure}[htb!]
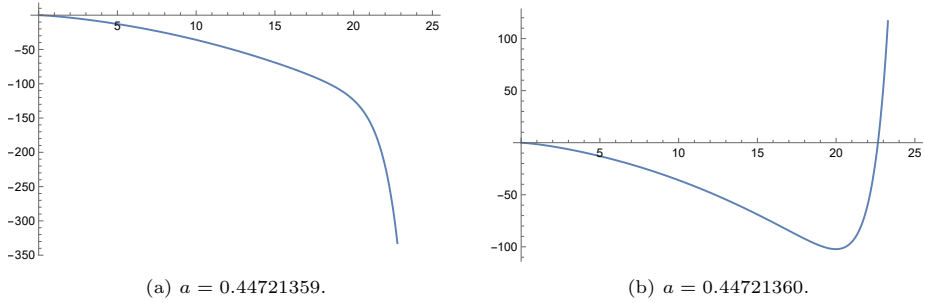

\begin{subfigure}{.5\textwidth}
\includegraphics[width=0.95\linewidth]{min1.png}
\caption{$a=0.44721359$.}
\end{subfigure}\quad
\begin{subfigure}{.5\textwidth}
\includegraphics[width=0.95\linewidth]{min2.png}
\caption{$a=0.44721360$.\label{fig:1b}}
\end{subfigure}
\caption{The minima $m(a,t)$ from (\ref{eq:min}) for two values $a$ close to $\frac{1}{\sqrt{5}} \approx 0.4472135955$ in the range $t\in [0,25]$.
Both figures were generated by \cite{mathematica}.\label{fig:2}}
\end{figure}
For $a = 0.44721359 < \frac{1}{\sqrt{5}}$ the minima $m(a,t)$ is always negative, i.e., $(e^{tA})_{t\geq 0}$ is not eventually positive.
For $a = 0.44721360 > \frac{1}{\sqrt{5}}$ the minima $m(a,t)$ changes sign and is in particular $>0$ for $t\geq \tau_a \approx 22.66$, i.e., $(e^{tA})_{t\geq 0}$ is eventually positive.

\begin{exms}\label{exm:1}
In the following we give for several $a\in\rset$ with $|a| > \frac{1}{\sqrt{5}}$ the approximate values of $\tau_{\pm a}$:
\begin{enumerate}[\bfseries\; (a)]
\item $a=0.44721360$ (\Cref{fig:1b}):
\[22.655 \quad < \quad \tau_{\pm a} \quad < \quad 22.656\]

\item $a=0.45$:
\[7.5504 \quad < \quad \tau_{\pm a} \quad < \quad 7.5505\]

\item $a=1$:
\[1.1675 \quad < \quad \tau_{\pm a} \quad < \quad 1.1676\]

\item $a=10$:
\[9.7541\cdot 10^{-2} \quad < \quad \tau_{\pm a} \quad < \quad 9.7542\cdot 10^{-2}\]

\item $a=100$:
\[9.6219\cdot 10^{-3} \quad < \quad \tau_{\pm a} \quad < \quad 9.6220\cdot 10^{-3} \tag*{$\circ$}\]
\end{enumerate}
\end{exms}

\section{Summary}
\label{sec:summary}

In \Cref{sec:Dfrechet} we showed that $(\fD,\,\cdot\,)$ is a regular Fr\'echet Lie group with Lie algebra $(\fd,\,\cdot\,,+)$.
This gave us a sufficiently analytic background to investigate $K$-positivity preservers and their generators for general closed subsets $K\subseteq\rset^n$.

In \Cref{sec:KposPres} we investigated $K$-positivity preservers.
At first we characterized in \Cref{thm:KposDescription} all $K$-positivity preservers for general closed subsets $K\subseteq\rset^n$.
We then treated $K$-positivity preservers with constant coefficients more deeply.
While a $K^\sharp$-moment sequence always gives a $K$-positivity preserver (\Cref{lem:sKsharpMomSeq}) we have that the reverse only holds for determinate sequences (\Cref{thm:determinateKsharp}) so far.
It is open if this also holds for the non-determinate case.

\begin{open}
In \Cref{lem:sKsharpMomSeq} we showed that
\begin{enumerate}[(i)]
\item $s$ is a $K^\sharp$-moment sequence
\end{enumerate}
implies that
\begin{enumerate}[(i)]\setcounter{enumi}{1}
\item $D(s)$ is a $K$-positivity preserver.
\end{enumerate}
The reverse holds for $K$ being a cone (\Cref{prop:KposPresCone}) and compact $K$ (\Cref{cor:compactK}).
In general it is open whether ``(ii) $\Rightarrow$ (i)'' holds.
Is there a counter example where it does not hold or a proof that is always holds?
\end{open}

In \Cref{cor:compactK} we showed that every $K$-positivity preserver with constant coefficients for compact $K$ is only the scaling $Tf = cf$ for some $c\geq 0$.

In \Cref{sec:generators} we characterized the generators $A$ of $K$-positivity preservers, i.e., all maps $A:\rset[x_1,\dots,x_n]\to\rset[x_1,\dots,x_n]$ such that $e^{tA}$ is a $K$-positivity preserver for all $t\geq 0$.
For that we need the regular Fr\'echet Lie group properties established in \Cref{sec:Dfrechet}.

At first we show in \Cref{prop:1plusAd} that the generators can always be fully characterized via the resolvent.
This characterization is well-known in the semi-group theory on Banach lattices.
However, the characterization through the resolvent gives no specific information for our case of the LF-space $\rset[x_1,\dots,x_n]$ and hence we use other techniques to gain more information about the generators.

In \Cref{thm:PosGeneratorsGeneral} we show that $A$ is a generator of a $\rset^n$-positivity preserving semi-group if and only if $A_y$ is a generator of a $\rset^n$-positivity preserving semi-group with constant coefficients for all $y\in\rset^n$.
The constant coefficient case was fully solved in \cite{didio24posPresConst} and hence we have explicit expressions how the coefficients of $A$ resp.\ $A_y$ look like.

Since for cones $K\subseteq\rset^n$ and the stripes $K = C\times [0,\infty)$ and $C\times\rset^n$ with $C\subseteq\rset^m$ compact we have that $s$ being a $K^\sharp$-moment sequence is necessary and sufficient we use the same technique via infinitely divisible representing measures from \cite{didio24posPresConst} here again for the constant coefficient case.
We get a complete description in \Cref{thm:generators0infty}.

For the non-constant coefficient case we have \Cref{prop:genK}, i.e., if $A_y$ is a generator of a $K$-positivity preserving semi-group for all $y\in K$ then $e^{tA}$ is a $K$-positivity preserver for all $t\geq 0$.
We give examples that the reverse of this statement does not hold in general.
Hence, it is still open how the generator of a $K$-positivity preserving semi-group look like in the non-constant coefficient case.

\begin{open}
How does
\[A:\rset[x_1,\dots,x_n]\to\rset[x_1,\dots,x_n]\]
look like if $e^{tA}$ is a $K$-positivity preserver for all $t\geq 0$ with $K\subsetneq\rset^n$ closed?
\end{open}

In \Cref{thm:sigma} we showed that while $A = (x\partial_x)^3$ is not a generator of a positivity preserver we have that $(e^{tA})_{t\geq 0}$ is an eventually positive semi-group, i.e., $e^{tA}$ is a positivity preserver for all $t\geq \tau$ for some $\tau>0$.
In \Cref{thm:1} we showed the same for $A = (x^2-1)\cdot\partial_x^2 + a\partial_x$ on $\rset[x]_{\leq 2}$.
Moments add therefore an additional tool to investigate eventually positive semi-groups \cite{daners16b,daners18,denkploss21,glueck23}.

\section*{Acknowledgment}

PJdD thanks Jochen Gl\"uck for providing us with the reference \cite{glueck23} and motivating us to look for an example of an eventually positive semi-group on $\rset[x]_{\leq 2}$.
This example is given in \Cref{thm:1}.

\section*{Funding}

The author and this project are supported by the Deutsche Forschungs\-gemein\-schaft DFG with the grant DI-2780/2-1 and his research fellowship at the Zukunfts\-kolleg of the University of Konstanz, funded as part of the Excellence Strategy of the German Federal and State Government.

The present work was performed during the stay of the second author at the University of Konstanz from January to March 2024.
This stay was funded by the Zukunftskolleg of the University of Konstanz.


\begin{thebibliography}{{Wol}21}

\bibitem[Bor11]{borcea11}
J.~Borcea, \emph{Classification of linear operators preserving elliptic,
  positive and non-negative polynomials}, J.~reine angew.\ Math. \textbf{650}
  (2011), 67--82.

\bibitem[CdDKM]{curtoHeat23}
R.~Curto, P.~J. di~Dio, M.~Korda, and V.~Magron, \emph{Time-dependent moments
  from partial differential equations and the time-dependent set of atoms},
  arXiv:2211.04416.

\bibitem[dD24a]{didio24hadamardArXiv}
P.~J. di~Dio, \emph{The {H}adamard {P}roduct of {M}oment {S}equences},
  arXiv:2407.19933.

\bibitem[dD24b]{didio24posPresConst}
\bysame, \emph{On {P}ositivity {P}reservers with constant {C}oefficients and
  their {G}enerators}, J.~Alg. (2024), in press,
  https://doi.org/10.1016/j.jalgebra.2024.07.050.

\bibitem[DG18]{daners18}
D.~Daners and J.~Gl\"{u}ck, \emph{A {C}riterion for the {U}niform {E}ventual
  {P}ositivity of {O}perator {S}emigroups}, Integr.\ Equ.\ Oper.\ Theory
  \textbf{90} (2018), 46, 19 pages.

\bibitem[DGK16]{daners16b}
D.~Daners, J.~Gl\"{u}ck, and J.~B. Kennedy, \emph{Eventually and asymptotically
  positive semigroups on {B}anach lattices}, J.~Diff.\ Eq. \textbf{261} (2016),
  2607--2649.

\bibitem[DKP21]{denkploss21}
R.~Denk, M.~Kunze, and D.~Plo{\ss}, \emph{The {B}i-{L}aplacian with {W}entzell
  {B}oundary {C}onditions on {L}ipschitz {D}omains}, Integr.\ Equ.\ Oper.\
  Theory \textbf{93} (2021), 13, 26 pages.

\bibitem[GH23]{glueck23}
J.~Gl\"uck and J.~H\"olz, \emph{Eventual cone invariance revisited}, J.~Alg.\
  Appl. \textbf{675} (2023), 274--293.

\bibitem[GS08]{guterman08}
A.~Guterman and B.~Shapiro, \emph{On linear operators preserving the set of
  positive polynomials}, J. Fixed Point Theory Appl. \textbf{3} (2008),
  411--429.

\bibitem[Hav36]{havila36}
E.~K. Haviland, \emph{On the momentum problem for distribution functions in
  more than one dimension {II}}, Amer.\ J.\ Math. \textbf{58} (1936), 164--168.

\bibitem[Kle06]{klenkewtheorie}
A.~Klenke, \emph{Wahrscheinlichkeitstheorie}, Springer, Berlin, 2006.

\bibitem[Mar08]{marshallPosPoly}
M.~Marshall, \emph{Positive {P}olynomials and {S}ums of {S}quares},
  Mathematical Surveys and Monographs, no. 146, American Mathematical Society,
  Rhode Island, 2008.

\bibitem[MO86]{miyaji86}
S.~Miyajima and N.~Okazawa, \emph{Generators of positive ${C}_0$-{S}emigroups},
  Pacific J.\ Math. \textbf{125} (1986), 161--175.

\bibitem[Net10]{netzer10}
T.~Netzer, \emph{Representation and {A}pproximation of {P}ositivity
  {P}reservers}, J.\ Geom.\ Anal. \textbf{20} (2010), 751--770.

\bibitem[Omo97]{omori97}
H.~Omori, \emph{Infinite-{D}imensional {L}ie {G}roups}, Translation of
  Mathematical Monographs, no. 158, American Mathematical Society, Providence,
  Rhode Island, 1997.

\bibitem[Ouh05]{ouhabaz05}
E.~M. Ouhabaz, \emph{Analysis of {H}eat {E}quations on {D}omains}, Princeton
  Universtiy Press, Princeton, Oxford, 2005.

\bibitem[Sch17]{schmudMomentBook}
K.~Schm\"{u}dgen, \emph{The {M}oment {P}roblem}, Springer, New York, 2017.

\bibitem[Sch23]{schmed23}
A.~Schmeding, \emph{An {I}ntroduction to {I}nfinite-{D}imensional
  {D}ifferential {G}eometry}, Cambridge studies in advanced mathematics, no.
  202, Cambridge Universtiy Press, Cambridge, UK, 2023.

\bibitem[SW99]{schaef99}
H.~H. Schaefer and M.~P. Wolff, \emph{Topological {V}ector {S}paces}, 2nd ed.,
  Springer-Verlag, New York, 1999.

\bibitem[Tr{\`e}67]{treves67}
F.~Tr{\`e}ves, \emph{Topological {V}ector {S}paces, {D}istributions and
  {K}ernels}, Academic Press, New York, 1967.

\bibitem[Wid72]{widder72}
D.~V. Widder, \emph{The {L}aplace {T}ransform}, Princeton University Press,
  Princeton, 1972, 8th printing.

\bibitem[{Wol}21]{mathematica}
{Wolfram Research{,} Inc.}, \emph{Mathematica, {V}ersion 13.0.0}, 2021,
  Champaign, IL.

\end{thebibliography}

\providecommand{\bysame}{\leavevmode\hbox to3em{\hrulefill}\thinspace}
\providecommand{\MR}{\relax\ifhmode\unskip\space\fi MR }
\providecommand{\MRhref}[2]{%
  \href{http://www.ams.org/mathscinet-getitem?mr=#1}{#2}
}
\providecommand{\href}[2]{#2}

\end{document}